\documentclass[twoside,11pt]{article}

\usepackage{jmlr2e}
\usepackage{jcdChange}

\usepackage{enumerate}
\usepackage{amsbsy}
\usepackage{amsmath}
\usepackage{amsfonts}
\usepackage{latexsym}
\usepackage{graphicx}
\usepackage{amssymb}
\usepackage{paralist}
\usepackage{tikz}
\usetikzlibrary{shapes,decorations}
\usepgflibrary{shapes.geometric}
\usepackage{amsmath}
\usepackage{subfigure}
\usepackage{multirow}
\usepackage{algorithmic}
\usepackage[ruled,vlined,boxed]{algorithm2e}
\usepackage{pst-grad} % For gradients
\usepackage{booktabs}
\usepackage{arydshln}
\usepackage[colorlinks=false,allbordercolors={1 1 1}]{hyperref}

\newtheorem{remark}[theorem]{Remark}
\let\oldremark\remark
\renewcommand{\remark}{\oldremark\normalfont}

\let\oldexample\example
\renewcommand{\example}{\oldexample\normalfont}

\firstpageno{1}

\allowdisplaybreaks

\long\def\symbolfootnote[#1]#2{\begingroup%
 \def\thefootnote{\fnsymbol{footnote}}\footnote[#1]{#2}\endgroup}

\renewcommand\eqref[1]{(\ref{#1})}

\begin{document}

%\title{ $\SW$: High-Dimensional Sparse Regression \\
%from Correlated Measurements}

\title{Swapping Variables for High-Dimensional Sparse Regression \\
with Correlated Measurements}

%\author{\name Divyanshu Vats \email dvats@rice.edu \\
%       \addr Department of Electrical and Computer Engineering\\
%       Rice University\\
%       Houston, TX 77251, USA
%       \AND
%       \name Richard G. Baraniuk \email richb@ece.rice.edu \\
%       \addr Department of Electrical and Computer Engineering\\
%       Rice University\\
%       Houston, TX 77251, USA}

\editor{}

\maketitle

\vspace{-1.5cm}
\begin{center}
Divyanshu Vats and Richard G. Baraniuk \\
Rice University \\
\{dvats, richb\}@rice.edu
\end{center}

\bigskip

\sloppy 
\begin{abstract}%
We consider the high-dimensional sparse linear regression problem of accurately estimating a sparse vector using a small number of linear measurements that are contaminated by noise.  It is well known that the standard cadre of computationally tractable sparse regression algorithms---such as the Lasso, Orthogonal Matching Pursuit (OMP), and their extensions---perform poorly when the measurement matrix contains highly correlated columns.  To address this shortcoming, we develop a simple greedy algorithm, called $\SW$, that iteratively \textit{swaps} variables until convergence.  $\SW$ is surprisingly effective in handling measurement matrices with high correlations.  In fact, we prove that $\SW$ outputs the true support, the locations of the non-zero entries in the sparse vector, under a relatively mild condition on the measurement matrix.  Furthermore, we show that $\SW$ can be used to boost the performance of \textit{any} sparse regression algorithm.  We empirically demonstrate the advantages of $\SW$ by comparing it with several state-of-the-art sparse regression algorithms.  
%Finally, we also analyze a natural extension of $\SW$ that swaps multiple variables, and thereby results in improved performance, at the cost of more computations.
\end{abstract}
\fussy


%\tableofcontents

% !TEX root = SWAPJournal.tex

\section{Introduction}

Many machine learning and statistics applications involve recovering a high-dimensional sparse (or approximately sparse) vector $\beta^*$ given a small number of linear observations $y = X \beta^* + w$, where $X$ is a known measurement matrix and $w$ is the observation noise.  Depending on the problem of interest, the unknown sparse vector $\beta^*$ could encode relationships between genes \citep{segal2003regression}, power line failures in massive power grid networks \citep{ZhuPower2012}, sparse representations of signals \citep{candes2006robust,CameraArray}, or edges in a graphical model \citep{MeinshausenBuhl2006,ravikumar2010high}, to name just a few applications.  
The simplest, but still very useful, setting is when the observations can be approximated by a \textit{sparse} linear combination of the columns of a measurement matrix $X$ weighted by the non-zero entries of the unknown sparse vector.  Sparse regression algorithms can be used to estimate the sparse vector, and subsequently the location of the non-zero entries.  

In this paper, we study sparse regression in a setting where current state-of-the-art methods falter or fail.  One of the key reasons why current methods fail is because of \textit{high correlations} between the columns of the measurement matrix $X$.  For example, if there exists a column $X_i$ that is nearly linearly dependent on the columns indexed by the locations of the non-zero entries in $\beta^*$, then many sparse regression algorithms will falsely select~$X_i$. 

There are many situations where it is natural for $X$ to contain correlated columns.  In signal processing, certain signals may admit a sparse representation in a basis, where the  basis elements (columns of $X$) can be significantly correlated with each other \citep{elad2006image}.  Such sparse representations are useful in signal processing tasks including denoising and compression \citep{elad2010sparse}.  In functional magnetic resonance imaging (fMRI) data, measurements from neighboring voxels can be significantly correlated \citep{varoquaux2012small}.  This can lead to inaccurate understanding of the connectivity between regions of the brain.  In gene expression data analysis, expression values of genes that are in the same pathway may be significantly correlated with each other \citep{segal2003regression}.    For example,  Figure~\ref{fig:correlatedeg} shows the pairwise correlations in three popular gene expression datasets.  Each pixel in the image is the absolute value of the normalized inner product of one gene expression with another gene expression, i.e.,   $\frac{|X_i^T X_j|}{\|X_i\|_2 \| X_j\|_2}$.  The higher/lower pixel intensities correspond to genes that have higher/lower pairwise correlations.  We clearly see that all three examples contain a large number of high pixel intensities.  This means that the gene expression values are highly correlated with each other, which may to lead to inaccurate identification of genes that are most relevant to understanding a disease.

\begin{figure}
\begin{center}
\includegraphics[scale=1.0]{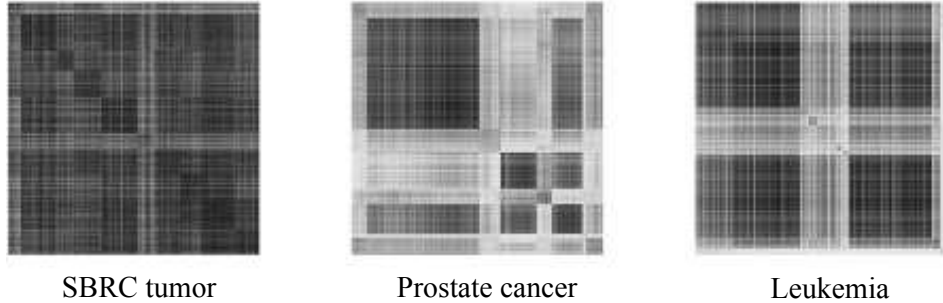}
\end{center}
\vspace{-0.5cm}
\caption{Pairwise correlations in gene expression data.  For gene expression data $X$, each pixel in the figure is equal to $\frac{|X_i^T X_j|}{\|X_i\|_2 \| X_j\|_2}$, where $X_i$ and $X_j$ refer to the columns of $X$.  The left, middle, and right figures correspond to pairwise correlations in gene expressions from patients with small round blue cell (SRBC) tumor, prostate cancer, and  leukemia, respectively.  Higher/lower pixel values corresponds to the correlation being close to $1/0$.}
\label{fig:correlatedeg}
\end{figure}
\fussy

\subsection{$\SW$: Overview}

In this paper, we develop and analyze $\SW$, a simple greedy algorithm for sparse regression with correlated measurements.  The input to $\SW$ is an estimate of the support of the unknown sparse vector $\beta^*$, i.e., the location of its non-zero entries.  The main idea behind $\SW$ is to iteratively perturb the estimate of the support by swapping variables.  The swapping is done in such a way that a loss function is minimized in each iteration of $\SW$.  In this way, $\SW$ seeks to estimate a support, in a greedy manner, that minimizes a loss function.  The main reason why $\SW$ is able to handle correlations is because even if an intermediate estimate contains a variable that is not in the true support, we show that $\SW$ can swap this variable with a true variable under relatively mild conditions on $X$.

As an example, suppose we initialize $\SW$ using a set $S^{(1)}$.  We note that $S^{(1)}$ could either be a random subset or it could be the support estimated by some other sparse regression algorithm. As we discuss later, selecting $S^{(1)}$ to be the output of a sparse regression algorithm has both statistical and computational advantages.  Starting with the support $S^{(1)}$, as illustrated in Figure~\ref{fig:swap}, $\SW$ iterates to an intermediate support $S^{(2)}$ by swapping a variable $i \in S^{(1)}$ with a variable $i' \in (S^{(1)})^c$ such that $S^{(2)} = \{S^{(1)} \backslash i\} \cup \{i'\}$.  The swapping is performed only if $S^{(2)}$ is superior to $S^{(1)}$ in terms of a given loss function and continues until convergence of the loss function.  The $\SW$ iterations can be summarized as follows:
\[
S^{(1)} \longrightarrow S^{(2)} = \{S^{(1)} \backslash i\} \cup \{i'\}
\longrightarrow S^{(3)} = \{S^{(2)} \backslash i\} \cup \{i'\} \longrightarrow \cdots \cdots \longrightarrow S^{(r)} \,.
\]

Naturally, the choice of $S^{(1)}$ plays an important role in the performance of $\SW$.  In particular, we prove the following two results:
\begin{enumerate}[({P}1)]
\item When $S^{(1)}$ misses at most one entry from the true support, $\SW$ outputs the true support under very mild conditions (see Theorem~\ref{thm:EqualCaseOne}).
\item When $S^{(1)}$ misses more than one entry from the true support, $\SW$ outputs the true support support under conditions that depend on certain correlations among the columns of $X$ (see Theorems~\ref{thm:EqualCaseTwo}, \ref{thm:EqualCaseThree}, and \ref{thm:optimalsample}).
\end{enumerate}

\begin{figure}
\centering
% Generated with LaTeXDraw 2.0.8
% Thu Feb 06 16:15:05 CST 2014
% \usepackage[usenames,dvipsnames]{pstricks}
% \usepackage{epsfig}
% \usepackage{pst-grad} % For gradients
% \usepackage{pst-plot} % For axes
\scalebox{0.7} % Change this value to rescale the drawing.
{
\begin{pspicture}(0,-0.96666014)(20.96,0.9266602)
\definecolor{color8b}{rgb}{0.8,0.8,1.0}
\psframe[linewidth=0.04,dimen=outer](8.72,0.90666014)(0.0,0.12666015)
\psline[linewidth=0.04cm,arrowsize=0.05291667cm 2.0,arrowlength=1.4,arrowinset=0.4,doubleline=true,doublesep=0.12]{->}(9.59,0.48666015)(11.59,0.46666014)
\psline[linewidth=0.04cm](4.2,0.88666016)(4.2,0.14666015)
\psline[linewidth=0.04cm](4.78,0.88666016)(4.78,0.14666015)
\usefont{T1}{ptm}{m}{n}
\rput(4.469092,0.53666013){\LARGE $i'$}
\psbezier[linewidth=0.04](1.4116406,0.046660155)(1.4116406,-0.43333983)(4.48,-0.43333983)(4.48,0.046660155)
\usefont{T1}{ptm}{m}{n}
\rput(2.9850976,-0.62333983){\LARGE swap}
\psframe[linewidth=0.04,linecolor=color8b,dimen=outer,fillstyle=solid,fillcolor=color8b](0.59164065,0.8666602)(0.04,0.16666016)
\psframe[linewidth=0.04,linecolor=color8b,dimen=outer,fillstyle=solid,fillcolor=color8b](1.6716406,0.8666602)(1.12,0.16666016)
\psframe[linewidth=0.04,linecolor=color8b,dimen=outer,fillstyle=solid,fillcolor=color8b](3.1316407,0.8666602)(2.58,0.16666016)
\psframe[linewidth=0.04,linecolor=color8b,dimen=outer,fillstyle=solid,fillcolor=color8b](6.4316406,0.8666602)(5.88,0.16666016)
\psframe[linewidth=0.04,linecolor=color8b,dimen=outer,fillstyle=solid,fillcolor=color8b](8.687461,0.8666602)(7.6958203,0.16666016)
\usefont{T1}{ptm}{m}{n}
\rput(1.3690917,0.51666015){\LARGE $i$}
\psframe[linewidth=0.04,dimen=outer](20.96,0.9266602)(12.24,0.14666015)
\psline[linewidth=0.04cm](16.44,0.90666014)(16.44,0.16666016)
\psline[linewidth=0.04cm](17.02,0.90666014)(17.02,0.16666016)
\usefont{T1}{ptm}{m}{n}
\rput(16.699091,0.5566602){\LARGE $i$}
\psframe[linewidth=0.04,linecolor=color8b,dimen=outer,fillstyle=solid,fillcolor=color8b](12.83164,0.88666016)(12.28,0.18666016)
\psframe[linewidth=0.04,linecolor=color8b,dimen=outer,fillstyle=solid,fillcolor=color8b](13.91164,0.88666016)(13.36,0.18666016)
\psframe[linewidth=0.04,linecolor=color8b,dimen=outer,fillstyle=solid,fillcolor=color8b](15.37164,0.88666016)(14.82,0.18666016)
\psframe[linewidth=0.04,linecolor=color8b,dimen=outer,fillstyle=solid,fillcolor=color8b](18.67164,0.88666016)(18.12,0.18666016)
\psframe[linewidth=0.04,linecolor=color8b,dimen=outer,fillstyle=solid,fillcolor=color8b](20.927462,0.88666016)(19.93582,0.18666016)
\usefont{T1}{ptm}{m}{n}
\rput(13.6390915,0.53666013){\LARGE $i'$}
\end{pspicture} 
}
\caption{An illustration of one iteration of the $\SW$ algorithm.  The shaded region corresponds to the estimate, say $S^{(t)}$, of the unknown true support that we seek to estimate.  The $\SW$ algorithm swaps a variable $i$ in $S^{(t)}$ with a variable $i'$ not in $S^{(t)}$.}
\label{fig:swap}
\end{figure}
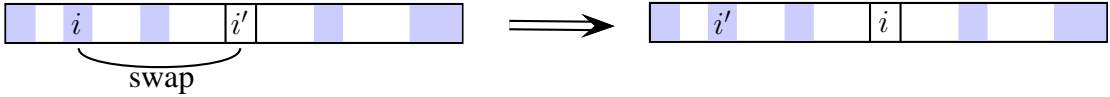

Property (P1) shows that $\SW$ can be used to post-process the output of any sparse regression algorithm without sacrificing for performance.  Property (P2) shows that $\SW$ can be used to estimate the true support of $\beta^*$ even when a sparse regression algorithm outputs a support set that differs significantly from the true support.  The particular condition in (P2), defined in Section~\ref{sec:theory}, highlights the role of $S^{(1)}$.  In particular, a sparse regression algorithm that can identify a larger fraction of the true support can potentially tolerate higher correlations among the columns of $X$.  

\begin{figure}[t]
\begin{center}
\vspace{-0.2cm}
\includegraphics[scale=0.8]{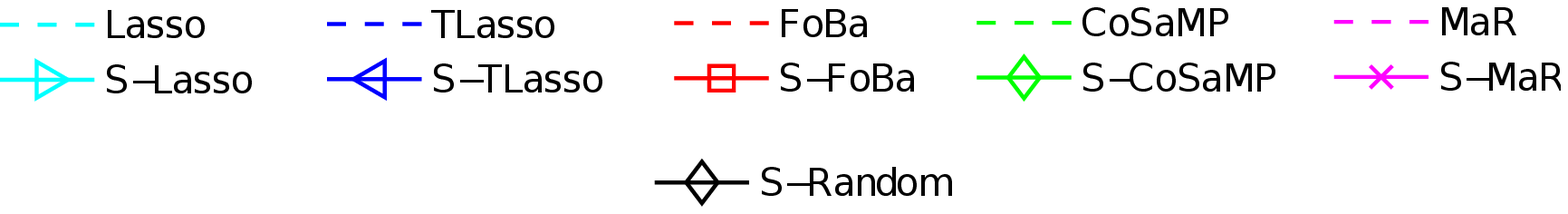}
\vspace{-0.7cm}
\end{center}
\begin{center}
\includegraphics[scale=0.46]{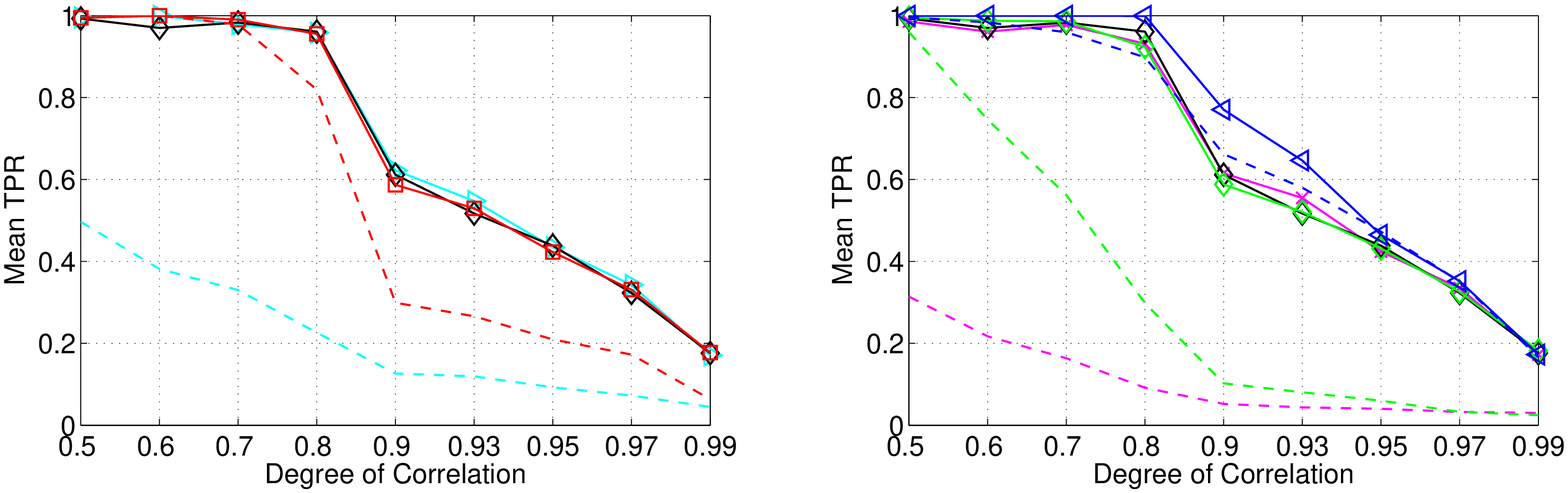}
\end{center}
\caption{ Mean true positive rate (TPR) versus the degree of correlation for several different sparse regression algorithms.  As the \textit{degree of correlation} increases, the amount of correlations among the columns of $X$ increases.  See method (E1) in Section~\ref{subsec:syn} for details about the experimental setup.  The dashed lines correspond to standard sparse regression algorithms, while the solid lines with markers correspond to $\SW$ based regression algorithms.  }
\label{fig:tpr_E1_Corr_Plot}
\end{figure}

We demonstrate the empirical performance of SWAP on synthetic and real gene expression data. In particular, we use several state-of-the-art sparse regression algorithms to initialize SWAP. For every initialization, we demonstrate that SWAP leads to improved support recovery performance.  For example, in Figure~\ref{fig:tpr_E1_Corr_Plot}, we plot the true positive rate (TPR), i.e. $|\widehat{S} \cap S^*|/k$, as the amount of correlation between the columns of $X$ increases.  See Section~\ref{subsec:syn} for details about the experimental setup.  The dashed lines in the plots correspond to a standard sparse regression algorithm, while the solid lines with markers correspond to a $\SW$ algorithm.
In most cases, we observe that the solid lines lie above the dashed line with the same color, which shows that $\SW$ is able to improve the performance of sparse regression algorithms.  Furthermore, we show that initializing $\SW$ using other sparse regression algorithms has computational benefits as opposed to initializing $\SW$ using a random subset.  In particular, if a sparse regression algorithm can partially estimate the support, then $\SW$ needs a smaller number of iterations to converge to the true support when compared to the case when $\SW$ is initialized with a random support.

\subsection{Related Work}

\sloppy
Current state-of-the-art computationally tractable sparse regression algorithms are only reliable under weak correlations among the columns of the measurement matrix.  These correlations are quantified either using the irrepresentability condition \citep{ZhaoYuLasso2006,tropp2007signal,MeinshausenBuhl2006,Wainwright2009} and/or various forms of the restricted eigenvalue (RE) conditions \citep{meinshausen2009lasso,bickel2009simultaneous}.  See \citet{buhlmann2011statistics} for a comprehensive review of such methods and the related conditions.

%Several methods have been proposed in the literature for sparse regression.  The theoretical properties of these methods either depend on the irrepresentability condition \citep{ZhaoYuLasso2006,tropp2007signal,MeinshausenBuhl2006,Wainwright2009} and/or various forms of the restricted eigenvalue (RE) conditions \citep{meinshausen2009lasso,bickel2009simultaneous}.  See \citet{buhlmann2011statistics} for a comprehensive review of such methods and the related conditions.  Both the irrepresentability condition and the RE conditions place limits on the correlations in the columns of the measurement matrix~$X$.

To the best of our knowledge, among the current methods in the literature, multi-stage methods \citep{wasserman2009high,meinshausen2009lasso,GeerBuhlZhouAdaptive2011,javanmard2013model} are most appropriate for handing correlated variables in regression.  An example of a multi-stage method is thresholded Lasso (TLasso), which first applies Lasso, a popular sparse regression algorithm proposed in \citet{TibshiraniLasso1994}, and then thresholds the absolute value of the estimated non-zero entries to output a final support.  Theoretically, TLasso requires an RE based condition for accurate support recovery that is stronger than the RE based condition required by $\SW$.  Moreover, TLasso, and other multi-stage methods, typically require more computations, since they require that computationally intensive model selection methods, such as cross-validation, be applied more than once to produce the final estimate.

When the columns in $X$ are highly correlated with other, exact support recovery may not be feasible, even with $\SW$.  In such cases, it is instead desirable to obtain a superset of the true support such that the superset is as small as possible.  Several methods have been proposed in the literature for estimating such a superset; see \citet{zou2005regularization, she2010,grave2011trace,huang2011sparse,buhlmann2012correlated} for some examples.  The main idea in these methods is to select all the highly correlated variables, even if only one of these variables is actually in the true support.  In its current form, $\SW$ is designed for exact support recovery in the presence of correlated measurements.  However, just as $\SW$ improves the performance of standard sparse regression algorithms for exact support recovery, we believe that suitable modifications of $\SW$ can improve the various superset estimation methods to deal with highly correlated measurements.

$\SW$ can be interpreted as a \textit{genetic algorithm} for solving a combinatorial optimization problem \citep{melanie1999introduction}.  In prior work, \citet{fannjiang2012coherence} empirically show the superior performance of a slightly different version of $\SW$ for handing correlated measurements.  However, no theory was given in \citet{fannjiang2012coherence} to understand its superior performance.  Our main contribution in this paper is to develop performance guarantees for $\SW$, and thereby demonstrate the advantages of using $\SW$ for sparse regression with correlated measurements.  A portion of the results in this paper appeared in \citet{VatsSwapNIPS}.

\fussy

\section{$\SW$: Problem Formulation}

In this section, we formulate the sparse regression problem and introduce the relevant notation and assumptions that will be used throughout the paper.  We assume that $y \in \R^n$, referred to as the \textit{observations}, and $X \in \R^{n \times p}$, referred to as the \textit{measurement matrix}, are known and related to each other by the linear model
\begin{equation}
y = X \beta^* + w \,, \label{eq:linmodel}
\end{equation}
where $\beta^* \in \R^p$ is the \textit{unknown sparse vector} that we seek to estimate and $w$ is measurement noise.  Unless mentioned otherwise, we assume the following throughout this paper:
\begin{enumerate}[({A}1)]
\item  The vector $\beta^*$ is $k$-sparse with the location of the non-zero entries given by the set $S^*$.  It is common to refer to $S^*$ as the \textit{support} of $\beta^*$, and we adopt this notation throughout the paper.  Furthermore, we refer to a variable in the support as an \textit{active} variable and refer to a variable outside the support as an \textit{inactive} variable.
\item The matrix $X$ is fixed with normalized columns, i.e., $\|X_i\|_2^2/n = 1$ for all $i \in [p]$, where $[p] = \{1,2,\ldots,p\}$.  In practice, normalization can easily be done by scaling $X$ and $\beta^*$ accordingly. 
\item The entries of $w$ are i.i.d.\ zero-mean sub-Gaussian random variables with parameter $\sigma$ so that $\E[\exp(t w_i)] \le \exp(t^2 \sigma^2/2)$.  The sub-Gaussian condition on $w$ is common in the literature and allows for a wide class of noise models, including Gaussian, symmetric Bernoulli, and bounded random variables \citep{VershyninRMT}.
\item The number of observations $n$ is greater than or equal to $k$, i.e., $n \ge k$.  As we shall see later, this assumption is important to compute the estimates of $\beta^*$.
\item The number of observations $n$ and the sparsity level $k$ are all allowed to grow to infinity as $p$ goes to infinity.  In the literature, this is referred to as the \textit{high-dimensional framework}.
\end{enumerate}
For any set $S$, we associate a loss function, ${\cal L}(S;y,X)$, that represents the cost associated with estimating $S^*$ by the set $S$.  An appropriate loss function for the linear problem in (\ref{eq:linmodel}) is the least-squares loss, which is given by 
\begin{align}
{\cal L}(S;y,X) \defn \min_{\alpha \in \R^{|S|}} \|y - X_{S} \alpha \|_2^2 = \bn \Pi^{\perp}[S] y \bn_2^2 \,, \label{eq:ls}
\end{align}
where $X_S$ refers to the $n \times |S|$ matrix that only includes the columns indexed by $S$ and $\Pi^{\perp}[S] = I - \Pi[S] = I - X_S ( X_S^T X_S)^{-1} X_S^T$ is the orthogonal projection onto the kernel of the matrix $X_S$. 
Assumption~(A4) is required so that the inverse of $X_S^T X_S$ exists.
Throughout this paper, we mainly study the problem of estimating $S^*$, since, once $S^*$ has been estimated, an estimate of $\beta^*$ can be easily computed by solving a constrained least-squares problem.  The subsequent error in estimating $\beta^*$ is given by following proposition.
\begin{proposition}
\label{thm:basicl2}
Let $\widehat{S}$ be an estimate of $S^*$ such that $\Pr(\widehat{S} \ne S^*) \le e^{-f(n,p,k)}$.  For any $ \tau_n \ge 1$, the constrained least-squares estimator, $\widehat{\beta} = (X_{\widehat{S}}^T X_{\widehat{S}})^{-1} X_{\widehat{S}}^T y$, satisfies the following bound with probability at least $ 1 - e^{-\tau_n - f(n,p,k)}$:
\begin{equation}
\| \widehat{\beta} - \widehat{\beta}^* \|_2^2 
\le c k \sigma^2 \tau_n/ (n \rho_{k}^2) \,, \label{eq:betal2bound}
\end{equation}
where $c$ is a universal positive constant and $\rho_k$ is the minimum eigenvalue of $X_{S^*}^T X_{S^*}/n$.
\end{proposition}
Choosing $\tau_n = \log n$, Proposition~\ref{thm:basicl2} shows that if the support $S^*$ can be estimated accurately with high probability, then the $\ell_2$-error in estimating $\beta^*$ can be upper bounded, with high probability, by ${c k \sigma^2 \log n}/{(n \rho_{k}^2)}$.  Clearly, this bound converges to $0$ as long as $n > ck\sigma^2 \log n / \rho_k^2$.  The proof of Proposition~\ref{thm:basicl2} follows from standard analysis of sub-Gaussian random variables.  %, is given in Appendix~\ref{app:basicl2}.  
%We note that if the estimate $\widehat{S}$ is a superset of $S^*$, then the bound in (\ref{eq:betal2bound}) holds with $k$ replaced by $|\widehat{S}|$.

Having established a bound on the estimation error given the true support, we now solely focus on the problem of estimating the true support $S^*$.  A classical method is to seek a support $S$ that minimizes the loss ${\cal L}(S;y,X)$.   This method, in general, may require a search over an exponentially number of possible supports.  \textit{The main goal in this paper is to design an algorithm that searches the space of all possible supports in a computationally tractable manner}.  Furthermore, we are interested in establishing conditions under which the algorithm can accurately estimate $S^*$, and as a result $\beta^*$, under much broader conditions on the measurement matrix $X$ than those imposed by current state-of-the-art methods.

The rest of the paper is organized as follows.  Section~\ref{sec:overview} presents the $\SW$ algorithm.  Section~\ref{sec:theory} analyzes $\SW$ and proves conditions under which $\SW$ leads to accurate support recovery.  Section~\ref{sec:sims} presents numerical simulations.  Section~\ref{sec:summary} concludes the paper.
%\noindent
%Notations: Before proceeding ahead, we summarize some notations used throughout this paper.  For a positive integer $q$, $[q]$ is the set $\{1,2,\ldots,p\}$.  For a vector $\theta$, (i) $\|\theta\|_0$ is the $\ell_0$ norm that counts the number of non-zero entries in $\theta$, (ii) $\supp(\theta)$ is the location of the non-zero entries in $\theta$.

\section{$\SW$: Detailed Description}
\label{sec:overview}
%\IncMargin{1em}
\begin{algorithm}[t]
\caption{$\SW(y,X,S)$}
\label{alg:MainAlg}
\smallskip
\hspace{-1em} {\textit{Inputs:} Observations $y$, measurement matrix $X$, and initial support $S^{(1)}$.} \\
\nl Let $r = 1$ and $L^{(1)} = {\cal L}(S^{(1)};y,X)$\\
\nl Swap $i \in S^{(r)}$ with $i' \in (S^{(r)})^c$ and compute the loss 
$L^{(r)}_{i,i'} = L(\{S^{(r)}\backslash i\} \cup i';y,X)$. \\
%estimate $\widehat{\beta}^{(i,i',r)}$ with support $\{ S^{(r)} \backslash i \} \cup \{i'\}$. \\
\nl \eIf{$\min_{i,i'} {\cal L}^{(r)}_{i,i'} < L^{(r)}$}{
\nl $\{\widehat{i},\widehat{i'}\} = \argmin_{i,i'} {\cal L}^{(r)}_{i,i'}$  (In case of a tie, choose a pair arbitrarily) \\
\nl Let $S^{(r+1)} = \{ S^{(r)} \backslash  \widehat{i} \}  \cup  \widehat{i'} $ and $L^{(r+1)}$ be the corresponding loss. \\
\nl Let $r = r + 1$ and repeat steps 2-4.
}{
\nl Return $\widehat{S} = S^{(r)}$.
}
\end{algorithm}

In this section, we describe the $\SW$ algorithm to find a support that minimizes a loss function. Suppose that we are given an estimate, say $S^{(1)}$, of the true support and let $L^{(1)} = {\cal L}(S^{(1)};y,X)$ be the corresponding least-squares loss (see (\ref{eq:ls})).  We want to transition to another estimate $S^{(2)}$ that is closer (in terms of the number of true variables), or equal, to $S^*$.  The $\SW$ algorithm transitions from $S^{(1)}$ to an $S^{(2)}$ in the following manner:
\[
\text{\textit{Swap every $i \in S^{(1)}$ with $i' \in (S^{(1)})^c$ and compute the  loss $L^{(1)}_{i,i'} = {\cal L}\left( \{S^{(1)}\backslash i\} \cup i' ;y,X\right)$.}}
\]
If $\min_{i,i'} L^{(1)}_{i,i'} < L^{(1)}$, then there exists a support that has a lower loss than $L^{(1)}$.  Subsequently, we find $\{\widehat{i}, \widehat{i}'\} = \arg \min_{ i,i' } L^{(1)}_{i,i'}$ and let $S^{(2)} = \{ S^{(1)} \backslash  \widehat{i}  \} \cup \{ \widehat{i}' \}$.  We repeat the above steps to find a sequence of supports $S^{(1)}, S^{(2)},\ldots,S^{(r)}$, where $S^{(r)}$ has the property that $\min_{i,i'} L^{(r)}_{i,i'}  \ge L^{(r)}$.  In other words, we stop $\SW$ when perturbing $S^{(r)}$ by one variable increases or does not change the resulting loss.  These steps are summarized in Algorithm~\ref{alg:MainAlg}.  We next make several remarks regarding $\SW$.

\subsection{Selecting the Initial Support $S^{(1)}$}  

\begin{figure}
\begin{center}
\includegraphics[scale=0.69]{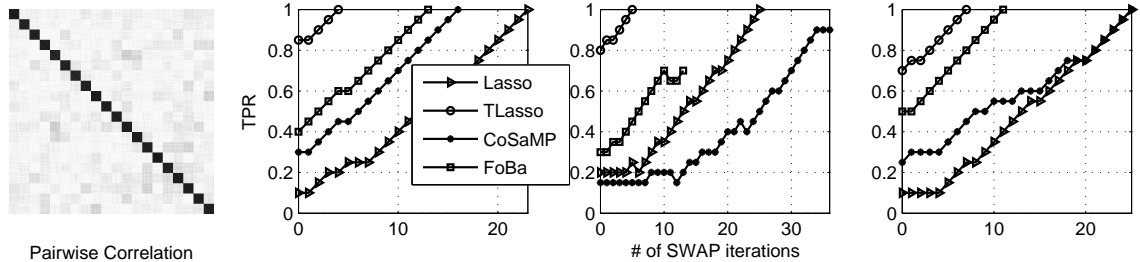}
\end{center}
\caption{An example that illustrates the performance of $\SW$ for a matrix $X$ with pairwise correlations given in the first figure.  The $(i,j)$-th pixel in the first figure is equal to $|X_i^T X_j|/n$.  Here, $p = 1000$, $n = 200$, $k = 20$, and $\sigma^2 = 1$.  Fixing $X$ and $\beta^*$, we generate \textit{three} different observations $y$ for three different realizations of the measurement noise $w$ in (\ref{eq:linmodel}).  The second, third, and fourth figures illustrate the performance of $\SW$ in the intermediate iterations for the three different realizations of the noise.  The horizontal axis is the number of $\SW$ iterations ($r$ in Algorithm~\ref{alg:MainAlg}) and the vertical axis is the true positive rate (TPR) of the intermediate estimates of the support.  }
\label{fig:swapempeg}
\end{figure}

The main input to $\SW$ is the initial support $S^{(1)}$, which also implicitly specifies the desired sparsity level of the estimated support.  Recall that $k$ is the unknown number of non-zero entries in $\beta^*$.  If $k$ is known, then $\SW$ can be initialized using the output of some other sparse regression algorithm.  In this way, $\SW$ can boost the performance of other sparse regression algorithms.

To illustrate the advantages of $\SW$, we set up a sparse regression problem with $p = 1000$, $n =200 $, $k = 20$, and $\sigma^2 = 1$, where $\sigma^2$ is the variance of the measurement noise $w$ in (\ref{eq:linmodel}).  The measurement matrix $X$ is chosen such that the pairwise correlations, i.e., the matrix $X^T X / n$, is given by the left most figure in Figure~\ref{fig:swapempeg}.  We simulate three different observation vectors for three different realizations of the noise $w$.  We use $\SW$ as a wrapper around four popular sparse regression algorithms: Lasso \citep{TibshiraniLasso1994}, thresholded Lasso (TLasso) \citep{GeerBuhlZhouAdaptive2011}, CoSaMP \citep{needell2009cosamp}, and FoBa \citep{zhang2011adaptive}.  The three plots in Figure~\ref{fig:swapempeg} show how the true positive rate (TPR) changes in the intermediate steps of the algorithm for the three different realizations of the noise, where TPR is the total number of active variables in an estimate divided by the total number of active variables.  For all the four different sparse regression algorithms, we observe that the TPR eventually increases when $\SW$ stops.  This demonstrates that $\SW$ is able to replace inactive variables in the initial estimates with active variables.  Furthermore, we generally observe that using $\SW$ with a support that contains more active variables is better for estimating the true support.  For example, TLasso is able to estimate more than half of the active variables accurately. Using $\SW$ with TLasso leads to accurate support recovery such that $\SW$ terminates after only a few iterations.  Our theoretical analysis in Section~\ref{sec:theory} sheds some light into this phenomenon.  In particular, we show that the sufficient conditions for accurate support recovery for $\SW$ become weaker as the number of active variables in the initial support $S$ increase.  

When $k$ is not known, which is the case in many applications, $\SW$ can be easily used in conjunction with other sparse regression algorithms to compute a solution path, i.e., a list of all possible estimates of the support over different sparsity levels.  Once a solution path is obtained, model selection methods, such as cross-validation or stability selection \citep{meinshausen2010stability}, can be applied to estimate the support.

\subsection{Computational Complexity} 
The main computational step in $\SW$ (Algorithm~\ref{alg:MainAlg}) is Line~2, where the loss $L_{i,i'}^{(r)}$ is computed for all possible swaps $(i,i')$.  If $s = |S^{(r)}|$, then clearly $s(p-s)$ such computations need to be done in each iteration of the algorithm.  Using properties of the orthogonal projection matrix (see Lemma~\ref{lemma:blockdecomp}) we have that for any $S$,
\begin{equation}
\Pi[S] = \Pi[S\backslash i] + \frac{(\Pi^{\perp}[S\backslash i] X_i)(\Pi^{\perp}[S\backslash i] X_i)^T}{X_i^T \Pi^{\perp}[S\backslash i] X_i} \,,\quad i \in S\,. \label{eq:rankone}
\end{equation}
To compute $L_{i,i'}^{(r)}$, we need to compute the orthogonal projection matrix $\Pi^{\perp}[\{S^{(r)} \backslash \{i\}\} \cup \{i'\}]$.  Once $\Pi^{\perp}[\{S^{(r)} \backslash \{i\}\}]$ is computed, $\Pi^{\perp}[\{S^{(r)} \backslash \{i\}\} \cup \{i'\}]$ can be easily computed for all $i' \in (S^{(r)})^c$ using the rank one update in (\ref{eq:rankone}).  Thus, effectively, the computational complexity of Line~2 is roughly $O(s(p-s) {\cal I}_{s-1})$, where ${\cal I}_{s-1}$ is the complexity of computing a projection matrix of rank $s-1$.  Using state-of-the-art matrix inversion algorithms \citep{coppersmith1990matrix}, ${\cal I}_{s} = O(s^{2.4})$.

There are several ways to significantly improve the computational complexity of $\SW$ at the expense of slightly degraded the performance.  One straightforward way, which was used in \citet{fannjiang2012coherence}, is to restrict the number of swaps using the pairwise correlations among the columns of $X$.  Another method is to first estimate a superset of the support, using methods such as in \citet{buhlmann2012correlated,VatsMuG2012}, and then restrict the swaps to only lie in the estimated superset.  Since the main goal in this paper is to study the statistical properties of $\SW$, we will address the computational aspects of $\SW$ in future work.

\subsection{Swapping Several Variables} 
A natural generalization of $\SW$ is to swap a group of $m$ variables with another group of $m$ variables.  The computational complexity of generalized $\SW$, which we refer to as $\SW^m$, is roughly $O\left( \binom{s}{m} \binom{p-s}{m} {\cal I}_{s-m}\right)$, where ${\cal I}_{s-m}$ is the complexity of computing a rank $s-m$ projection matrix.  Clearly, for $m$ large, $\SW^m$ is not tractable in its na\"{i}ve form.  In particular, as $m$ increases, the complexity of $\SW^m$ approaches the complexity of the exponential search algorithm that searches among all possible supports of size $s$.

\subsection{Comparison to Other Algorithms}

$\SW$ differs significantly from other greedy algorithms for sparse regression.  When $k$ is known, the main distinctive feature of $\SW$ is that \textit{it always maintains a $k$-sparse estimate of the support}.  Note that the same is true of the computationally intractable exhaustive search algorithm.  Other competing algorithms, such as forward-backward (FoBa) \citep{zhang2011adaptive} and CoSaMP \citep{needell2009cosamp}, usually estimate a sparse vector with higher sparsity level and iteratively remove variables until $k$ variables are selected.  The same is true for multi-stage algorithms \citep{zhang2009some,wasserman2009high, ZhangMulti2010,GeerBuhlZhouAdaptive2011}.  Intuitively, as we shall see in Section~\ref{sec:theory}, by maintaining a support of size $k$, the performance of $\SW$ only depends on correlations among the columns of the matrix $X_{A}$, where $A$ is of size at most $2k$ and includes the true support.  In contrast, for other sparse regression algorithms, $|A| > 2k$. 
%In Figure~\ref{fig:pseudo}, we compare $\SW$ to several state of the art algorithms (see Section~\ref{sec:numsim} for a description of the algorithms).  \textit{In all cases, $\SW$ results in superior performance}.

\section{Theoretical Analysis of $\SW$}
\label{sec:theory}

In this section, we present the main theoretical results of the paper that identifies the conditions under which $\SW$ performs accurate support recovery.  Section~\ref{sec:esd} analyzes the performance of an exhaustive search decoder to understand the performance guarantees of sparse regression when given no restrictions on the computational complexity.  Section~\ref{subsec:impparam} defines important parameters that are used to state the main theoretical results in Section~\ref{subsec:mainr}.  Section~\ref{subsec:example} presents an example to highlight the advantages of using $\SW$.  Section~\ref{subsec:discussion} presents some general remarks regarding the main results and discusses some extensions.

\subsection{Exhaustive Search Decoder}
\label{sec:esd}

Recall that we seek to estimate the support set $S^*$ such that the observations $y$ is a linear combination of the columns in the matrix $X_{S^*}$.  Before presenting a detailed theoretical analysis of $\SW$, we first study the \textit{exhaustive search decoder} that estimates $S^*$ as follows: % and ${\beta^*}$ as follows:
\begin{align}
\text{(ESD): }& \widehat{S} =\min_{S \in \Omega_k} {\cal L}(S;y,X) \,, \nonumber
%&\widehat{\beta} =\left(X_{\widehat{S}}' X_{\widehat{S}}\right)^{-1} X_{\widehat{S}}^T y \,, %\nonumber 
\end{align}
where $\Omega_k = \{S: S \subset [p], |S| \le k\}$ is the set of all possible supports of size $k$.  Thus, the ESD method searches among \textit{all possible} supports of size $k$ to find the support that minimizes the loss.  In the literature, the ESD method is also referred to as the $\ell_0$-norm solution.  Understanding the performance of ESD is important to analyzing the performance of $\SW$, since, if $S^*$ does \textit{not} minimize the loss, then $\SW$ will most likely not estimate $S^*$ accurately.  
%
%Note that the parameter $s$ in ESD is a regularization parameter.  When $k$ is not known, which is mostly the case, $s$ must be selected using model selection methods such as cross-validation or stability selection.  In this case, it is necessary to apply the algorithm, ESD in this case, for multiple values of the regularization parameter and then select the parameter that maximizes (or minimizes) a score.  Since it is typical for model selection algorithms to overestimate the number of non-zero entries, we analyze $ESD$ when $s \ge k$ to understand how the support varies when $k$ is not estimated accurately.  Furthermore, this general analysis also paves way to conditions when ESD may not estimate $S^*$ accurately for $s = k$, but may be able to estimate a superset of $S^*$.  
%
Before stating the theorem, we define the following two parameters:
\begin{align}
\rho_{k+\ell} &\defn \inf \left\{ \frac{\|X \theta \|_2^2}{n \|\theta\|_2^2} : \|\theta\|_0 \le {k+\ell} \,, S^* \subseteq \supp(\theta) \right\} \label{eq:rhokl}  \,, \\
\beta_{\min} &\defn \min_{i \in S^*} |\beta_i| \,. \label{eq:betamin}
\end{align}
The parameter $\rho_{k+\ell}$ is the eigenvalue of certain diagonal blocks of the matrix $X^T X/n$ of size $k+\ell$ that includes the block $X_{S^*}^T X_{S^*}/n$.  The parameter $\beta_{\min}$ is the minimum absolute value of the non-zero entries in $\beta^*$.  Both $\rho_{k+\ell}$ (or its different variations) and 
$\beta_{\min}$ are known to be crucial in determining the performance of sparse regression algorithms.

\begin{proposition}
\label{thm:ESDFirst}
Consider the linear model in (\ref{eq:linmodel}) and suppose (A1)--(A5) holds.  Let $\widehat{S}$ be the ESD estimate with $s = k$ and let $\rho_{2k} > 0$.  If ${n > \frac{4 + \log(k^2 (p-k)) }{ c^2 \beta_{\min}^2 \rho_{2k} }}$, where $0 < c^2 \le 1/(18 \sigma^2)$, then $\Pr(\widehat{S} = S^*) \rightarrow 1$ as $p \rightarrow \infty$.
\end{proposition}
Proposition~\ref{thm:ESDFirst} specifies the scaling on the number of observations $n$ to ensure that ESD estimates the true support in the high-dimensional setting.  The proof of Proposition~\ref{thm:ESDFirst} is outlined in Appendix~\ref{app:ESDFirst}.  The steps in the proof mirror those of the ESD analysis in \citet{wainwright2009information}.  The main difference in the proof comes from the assumption that $X$ is fixed, as opposed to being sampled from a Gaussian distribution.  The dependence on $\rho_{2k}$ in Proposition~\ref{thm:ESDFirst} is particularly interesting, since, to the best of our knowledge, the performance of state-of-the-art computationally tractable methods for sparse regression depend either on certain correlations among the columns of $X$ or on the restricted eigenvalues $\rho_{k+s}$ for some $s > k$.  Since $\rho_{2k} > \rho_{k+s}$ for $s >k$, ESD leads to accurate support recovery for a much broader class of measurement matrices $X$.  However, ESD is computationally intractable, and so it is desirable to devise algorithms that are computationally tractable and offer similar performance guarantees.  %In the next section, we analyze $\SW$ and show how its performance nearly mirrors that of ESD.

\subsection{Some Important Parameters}
\label{subsec:impparam}

%\begin{figure}
%\caption{Illustration of the parameters that control the performance of $\SW$. (a) Restricted eigenvalue $\rho_{k+\ell}$, (b) Restricted eigenvalue $\rho_{k,d}$, (c) Minimum absolute value of the non-zero entries in $\beta^*$}
%\label{fig:parameters}
%\end{figure}

In this section, we collect some important parameters that determine the performance of $\SW$.  
%An illustration of all the parameters defined in this Section is given in Figure~\ref{fig:parameters}.  
We have already defined the minimum absolute value of the non-zero entries, $\beta_{\min}$, in (\ref{eq:betamin}) and the restricted eigenvalue (RE) $\rho_{k+\ell}$ in (\ref{eq:rhokl}).  Another form of the restricted eigenvalue used in our analysis is defined as follows:
\begin{align}
\rho_{k,\ell} &\defn \inf \left\{ \frac{\|X \theta \|_2^2}{n \|\theta\|_2^2} : \|\theta\|_0 \le {k} \,, |S^* \cap \supp(\theta)| \ge \ell \right\} \label{eq:rhokd} \,.
\end{align}
The parameter $\rho_{k,\ell}$ defined above is the minimum eigenvalue of certain blocks of the matrix $X^T X/n$ of size $k$ that includes the blocks $X^T_{A} X_{A}/n$, where $A$ is a subset of $S^*$ of size at least $\ell$.  It is clear that smaller values of $\rho_{k+\ell}$ or $\rho_{k,\ell}$ correspond to correlated columns in the matrix $X$.

Next, we define two parameters that characterize the correlations between the columns of the matrix $X_{S^*}$ and the columns of the matrix $X_{(S^*)^c}$, where recall that $S^*$ is the true support of the unknown sparse vector $\beta^*$.  For a set $\Omega_{k,d}$ that contains all supports of size $k$ with atleast $k-d$ active variables from $S^*$, define $\gamma_d$ as  
% ${\Omega}_{k,d} = \{ S: |S^*\backslash S| \le d, |S| = k  \}$ In our analysis, $d$ will represent the number of active variables in the support set that is used to initialized $\SW$.  
\begin{align}
{\gamma}_d \defn
\max_{S \in {\Omega_{k,d}} \backslash S^*} \min_{i \in (S^*)^c \cap S} 
\frac{\left\| \Sigma_{i,\bar{S}}^{S\backslash i}
\left(\Sigma_{\bar{S},\bar{S}}^{S\backslash i}\right)^{-1} \right\|_1^2}
{\Sigma_{i,i}^{S\backslash i}}  \,, \quad \bar{S} = S^* \backslash S \,,
\label{eq:gammaD}
\end{align}
where $\Sigma^{B} = X^T \Pi^{\perp}[B] X/n$.  The matrix $\Sigma^B$ is the pairwise correlations between vectors that are projected onto the kernel of the matrix $X_{B}$.  When $B = \emptyset$, we write $\Sigma = \Sigma^B$.
Popular sparse regression algorithms, such as the Lasso and the OMP, can perform accurate support recovery when $\zeta = \max_{i \in (S^*)^c} \| \Sigma_{i,S^*} \Sigma_{S^*,S^*}^{-1} \|_1^2 < 1$.  We will show in Section~\ref{subsec:mainr} that $\SW$ can perform accurate support recovery when $\gamma_d < 1$.  Although the form of $\gamma_d$ is similar to $\zeta$, there are several key differences, which we highlight as follows:
\begin{itemize}
\item Since $\Omega_{k,d}$ contains all supports such that $|S^*\backslash S| = d$, it is clear that $\gamma_d$ is the $\ell_1$-norm of a $d \times 1$ vector, where $d \le k$.  In contrast, $\zeta$ is the $\ell_1$-norm of a $k \times 1$ vector.  If indeed $\zeta < 1$, i.e., accurate support recovery is possible using the Lasso, then $\SW$ can be initialized by the output of the Lasso.  In this case, $\gamma_d = 0$ and $\SW$ also outputs the true support as long as $S^*$ minimizes the loss function.  We make this statement precise in Theorem~\ref{thm:EqualCaseOne}.  Thus, it is only when $\zeta \ge 1$ that the parameter $\gamma_d$ plays a role in the performance of $\SW$.

\item The parameter $\zeta$ directly computes correlations between the columns of $X$.  In contrast, $\gamma_d$ computes correlations between  the columns of $X$ when projected onto the null space of a matrix $X_B$, where $|B| = d-1$.  

\item Note that $\gamma_d$ is computed by taking a \textit{maximum} over supports in the set $\Omega_d \backslash S^*$ and a \textit{minimum} over inactive variables in each support.  The reason that the minimum appears in $\gamma_d$ is because we choose to swap variables that result in the minimum loss.  In contrast, $\zeta$ is computed by taking a \textit{maximum} over all inactive variables.  This minimum is what allows $\SW$ to tolerate higher correlations among columns of $X$ when compared to other sparse regression algorithms.

\end{itemize}

%We will show examples in Section~3.3 of computing $\gamma_d$ for different types of measurement matrices and compare $\gamma_d$ to $\zeta$. DO THIS SOON.  
In addition to characterizing the performance of $\SW$ using the parameter $\gamma_d$, we also make use of the following parameter:
\begin{align}
\nu_d \defn
\max_{S \in {\Omega_{k,d}} \backslash S^*} \min_{i}
\max_{j,j' } \frac{\left\| \Sigma_{i,\bar{S}_j}^{S\backslash \{i, j\}}
\left(\Sigma_{\bar{S}_j,\bar{S}_j}^{S \backslash \{i,j\}}\right)^{-1} \right\|_1^2 + \left\| \Sigma_{j',\bar{S}_j}^{S\backslash \{i,j\}}
\left(\Sigma_{\bar{S},\bar{S}_j}^{S\backslash \{i, j\}}\right)^{-1} \right\|_1^2}
{\min\left\{ \Sigma_{i,i}^{S\backslash \{i, j\}} , \Sigma_{j',j'}^{S\backslash \{i, j\}} \right\}} \label{eq:nud} \,,
\end{align}
where $i \in  (S^*)^c \cap S$, $j \in S$, $j' \in S^c \cap (S^*)^c$, and $\bar{S} = \{S \backslash i\} \cup \{j\}$.  We will see in Theorem~\ref{thm:optimalsample} that in the noiseless case, i.e., $\sigma = 0$, $\nu_d < 1$ ensures that $\SW$ only swaps an active variable with another active variable or swaps an inactive variable with another inactive variables.  
This enables us to show that the sample complexity of $\SW$ can be at par with that of the exhaustive search decoder.

\subsection{Statement of the Main Results}
\label{subsec:mainr}
In this section, we state the main theoretical results that characterize the performance of $\SW$.  A summary of the results in this section is given as follows:

\begin{itemize}

\item Theorem~\ref{thm:EqualCaseOne} (Section~\ref{subsubsec:one}): Identifies sufficient conditions for accurate support recovery when $\SW$ is initialized by a support that is either equal to the true support or differs from the true support by one variable.  This theorem motivates the use of $\SW$ as a method to boost the performance of any sparse regression algorithm.

\item Theorem~\ref{thm:EqualCaseTwo} (Section~\ref{subsubsec:two}): Identifies sufficient conditions for accurate support recovery when $\SW$ is used with a sparse regression algorithm that outputs a support that differs from the true support by more than one variable.

\item Theorem~\ref{thm:EqualCaseThree} (Section~\ref{subsubsec:three}): Weakens the sufficient conditions in Theorem~\ref{thm:EqualCaseTwo} by assuming certain properties of the initial support.

\item Theorem~\ref{thm:optimalsample} (Section~\ref{subsubsec:four}): Shows that $\SW$ can perform accurate support recovery under the same scaling on the number of observations as the ESD algorithm in Section~\ref{sec:esd}.

\end{itemize}

Throughout this Section, we assume that $\SW$ is initialized by a support $S^{(1)}$ of size $k$ and $\widehat{S}$ is the output of $\SW$.

\subsubsection{Initialization with the True Support $S^*$ }
\label{subsubsec:one}

We claim that $\SW$ can be used as a wrapper around other sparse regression algorithms to boost their performance.  However, for this to be true, it is important to show that $\SW$ does not spoil the results of other sparse regression algorithms.  

\begin{theorem}
\label{thm:EqualCaseOne}
Consider the linear model in (\ref{eq:linmodel}) and suppose (A1)--(A5) holds.  Let the input $S^{(1)}$ to $\SW$ be such that $|S^{(1)}| = k$ and $|S^* \backslash S^{(1)}| \le 1$. If 
\begin{equation}
n > \frac{4 + \log(k^2 (p-k)) }{ c^2 \beta_{\min}^2 \rho_{2k}/2 }, \label{eq:numNf}
\end{equation}
where $0 < c^2 \le 1/(18 \sigma^2)$, then $\P(\widehat{S} = S^*) \rightarrow 1$ as $p \rightarrow \infty$.
\end{theorem}
\begin{proof}
If $S^*$ minimizes the loss among all supports of size $k$, then $\SW$ clearly outputs $S^*$ when initialized with a support $S^{(1)}$ such that $|S^{(1)}| = 1$ and $|S^* \backslash S^{(1)}| = 1$.  From Theorem~\ref{thm:ESDFirst}, the conditions stated in the result guarantee that $S^*$ minimizes the loss with high probability.
\end{proof}
Theorem~\ref{thm:EqualCaseOne} states that if the input to $\SW$ falsely detects at most one variable, then $\SW$ is high-dimensional consistent when given a sufficient number of observations $n$.  In particular, the condition on $n$ in (\ref{eq:numNf}) is mainly enforced to guarantee that the true support $S^*$ minimizes the loss function.  This condition is weaker than the sufficient conditions required for computationally tractable sparse regression algorithms.  For example, the method FoBa is known to be superior to other methods such as the Lasso and the OMP.  \citet{zhang2011adaptive} show that FoBa requires $n = \Omega( \log(p)/ (\rho_{k+\ell}^3 \beta_{\min}^2))$ observations for high-dimensional consistent support recovery, where the choice of $\ell$, which is greater than $k$, depends on the correlations among the matrix $X$.  In contrast, the condition in (\ref{eq:numNf}), which reduces to $n =  \Omega( \log(p-k)/ (\rho_{2k} \beta_{\min}^2))$, is weaker since $1/\rho_{k+\ell}^3 < 1/\rho_{2k}$ for $\ell > k$ and $p -k < p$.  This shows that if a sparse regression algorithm can accurately estimate the true support, then $\SW$ does not introduce any false positives and also outputs the true support.  Furthermore, if a sparse regression algorithm falsely detects one variable, then $\SW$ can potentially recover the correct support.  Thus, we conclude that using $\SW$ can only improve the chances of recovering the true support.

\subsubsection{Initialization with an Arbitrary Support}
\label{subsubsec:two}

We now consider the more interesting case when $\SW$ is initialized by a support $S^{(1)}$ that falsely detects more than one variable.  In this case, $\SW$ will clearly need more than one iteration to recover the true support.  Furthermore, to ensure that the true support can be recovered, we need to impose some additional assumptions on the measurement matrix $X$.  The particular condition we enforce depends on the parameter $\gamma_k$ defined in (\ref{eq:gammaD}).  As mentioned in Section~\ref{subsec:impparam}, $\gamma_k$ captures the correlations between the columns of $X_{S^*}$ and the columns of $X_{(S^*)^c}$.  To simplify the statement in the next Theorem, define the function $g(\delta,\rho,c)$ as
\begin{equation}
g(\delta,\rho,c) = (\delta - 1) + 2c(\sqrt{\delta} + 1/\sqrt{\rho}) + 2c^2\,. \nonumber 
\end{equation}
\begin{theorem}
\label{thm:EqualCaseTwo}
Let the input to $\SW$ be such that $|S^{(1)}| = k$ and $|S^* \backslash S^{(1)}| > 1$.  If for a constant $c$ such that $0 < c^2 < 1/(18\sigma^2)$, $g(\gamma_k,\rho_{k,1},c\sigma) < 0$, $\log\binom{p}{k} > 4+\log(k^2(p-k))$, and $n > \frac{2\log\binom{p}{k}}{c^2 \beta_{\min}^2 \rho_{2k}^2}$,
then $\P(\widehat{S} = S^*) \rightarrow 1$ as $p \rightarrow \infty$.
\end{theorem}
Theorem~\ref{thm:EqualCaseTwo} says that if $\SW$ is initialized by \textit{any support} of size $k$, then $\SW$ outputs the true support in the high-dimensional setting as long as $\gamma_k$ and $n$ satisfy the conditions stated in the theorem.  It is easy to see that in the noiseless case, i.e., when $\sigma = 0$, the condition required for accurate support recovery reduces to $\gamma_k < 1$.  As discussed in Section~\ref{subsec:impparam}, there is reason to believe that this condition may be much weaker than the conditions imposed by other sparse regression algorithms.

The proof of Theorem~\ref{thm:EqualCaseTwo}, outlined in Appendix~\ref{app:EqualCaseTwo}, relies on imposing conditions on each support $S \in \Omega_k \backslash S^*$ such that that there exists a swap so that the loss can be decreased.  Clearly, if such a property holds for each support, except $S^*$, then $\SW$ will output the true support since (i) there are only a finite number of possible supports, and (ii) each iteration of $\SW$ results in a different support.  The dependence on $\binom{p}{k}$ in the expression for the number of observations $n$ arises from applying the union bound over all supports of size $k$.

\subsubsection{Initialization with the Output of a Sparse Regression Algorithm}
\label{subsubsec:three}

The condition in Theorem~\ref{thm:EqualCaseTwo} is independent of the initialization $S^{(1)}$, which is why the sample complexity, i.e., the number of observations $n$ required for consistent support recovery, scales as $\log \binom{p}{k}$.  To reduce the sample complexity, we can impose additional conditions on the support $S^{(1)}$ that is used to initialize $\SW$.  One such condition is to assume that $S^{(1)}$ has certain optimality conditions over a subset of the variables from the true support.  In particular, define the event ${\cal E}_{k,d}$ as
\begin{align}
%\Omega(S^{(1)}) &= \left\{ S: S = \{ S^{(1)} \cap S^*\} \cup T', |S| = k \right\} \,, \\
{\cal E}_{k,d} = \left\{ {\cal L}(S^{(1)};y,X) < \min_{S \in {\bar{\Omega}_{k,d}} \backslash S^{(1)}, |S| = k} {\cal L}(S;y,X) \right\} \,, \quad  |S^* \backslash S^{(1)}| = d \,,  \label{eq:ess1}
\end{align}
where $\bar{\Omega}_{k,d}^c = \{S: |S| = k, |S^*\backslash S| \ge d\}$ contains all supports of size $k$ that contain at most $k-d$ active variables from $S^*$.
The event ${\cal E}_{k,d}$ is the set of outcomes for which the loss associated with $S^{(1)}$ is less than the loss associated with all supports that contain a smaller or equal number of active variables than $S^{(1)}$.  If we assume that $\Pr({\cal E}_{S^{(1)}}) = 1$, then all iterations of $\SW$ will lie in the set  $\Omega_{k,d+1}$.  Furthermore, by a simple counting argument, $|\Omega_{k,d+1}| \le \binom{p}{d}^3$.  This leads to the following theorem.

%Note that the maximum in the definition of $\gamma$ in \eqref{eq:gammaC} is taken over all supports of size $k$, while the maximum in $\bar{\gamma}_d$ is taken over a subset of the variables in $\Omega_{k}$.  Thus, it is clear that $\bar{\gamma}_{d} \le \gamma$.  Furthermore, it is easy to see that $\gamma_0 \le \gamma_{1} \le \cdots \le \gamma_{k-1} \le \gamma_{k} = \gamma$.
\begin{theorem}
\label{thm:EqualCaseThree}
Let the input to $\SW$ be such that $|S^{(1)}| = k$, $|S^* \backslash S^{(1)}| = d > 1$, and $\Pr({\cal E}_{S^{(1)}}) \rightarrow 1$. If for a constant $c$ such that $0 < c^2 < 1/(18\sigma^2)$, $g(\gamma_{d-1},\rho_{k,1},c\sigma) < 0$, $3\log\binom{p}{d} > 4+\log(k^2(p-k))$, and $n > \frac{6\log\binom{p}{d}}{c^2 \beta_{\min}^2 \rho_{2k}^2}$,
then $\P(\widehat{S} = S^*) \rightarrow 1$ as $p \rightarrow \infty$.
\end{theorem}

The proof of Theorem~\ref{thm:EqualCaseThree} follows easily from the proof of Theorem~\ref{thm:EqualCaseTwo}.  The main consequence of Theorem~\ref{thm:EqualCaseThree} is that if a sparse regression algorithm can achieve consistent partial support recovery, then the conditions needed for support recovery using $\SW$ are weakened.  Moreover, as the number of active variables in $S^{(1)}$ increases, the sufficient conditions required for accurate support recovery using $\SW$ become weaker since $\gamma_{d-1} \le \gamma_k$ for $d > 1$.

\subsubsection{Achieving the optimal sample complexity}
\label{subsubsec:four}

One drawback of the theoretical analysis presented so far is that the conditions require that the number of observations $n$ scale as $\log\binom{p}{d}$, where $d \le k$.  The reason for the dependence on $d$ is that, once we assume that $\Pr({\cal E}_{S,d}) = 1$, there are $O(\binom{p}{d})$ possible number of supports that $\SW$ can encounter.  To ensure that $\SW$ does not make an error, we use the union bound over all these sets to bound the probability of making an error.  In practice, however, once the support set $S$ is fixed, the total number of possible supports that $\SW$ can visit can be much smaller than $O(\binom{p}{d})$.  The exact number of possible supports will depend on the correlations between the columns of $X$.  In the next theorem, we show that under additional assumptions on $X$, $\SW$ can achieve similar sample complexity as the exhaustive searcher decoder.

\begin{theorem}
\label{thm:optimalsample}
Let $S$ be the input to $\SW$ such that $|S| = k$ and 
$|S^* \backslash S| \le d$.  If for a constant $c$ such that $0 < c^2 < 1/(18\sigma^2)$, $g(\nu_d,\rho_{k-1,0}/2,c\sigma) < 0$, and $n > \frac{2k+\log(k(p-k))}{c^2\beta_{\min}^2 \rho_{2k}^2/4}$,
then $\P(\widehat{S} = S^*) \rightarrow 1$ as $(n,p,k) \rightarrow \infty$.
\end{theorem}
The proof of Theorem~\ref{thm:optimalsample} is similar to the proof of Theorem~\ref{thm:EqualCaseTwo} and is outlined in Appendix~\ref{app:optimalsample}.  The condition $g(\nu_d,\rho_{k-1,0}/2,c\sigma) < 0$ ensures that, when $\SW$ is initialized with an appropriate support $S$ of size $k$, then the $\SW$ iterations only swap an active variable with an active variable or swap an inactive variable with an active variable.  This allows us to upper bound the total number of possible supports that $\SW$ can visit from $O\left(\binom{p}{d}\right)$ to $2^k$.  We note that in numerical simulations, we observed that even when the $\SW$ iterations swapped an inactive variable with an inactive variable, $\SW$ performed accurate support recovery.  Thus, not allowing an inactive variable to be swapped with an inactive variable is rather restrictive.  We believe that that a more involved analysis, that allows a constant number of swaps of an inactive with an inactive, can further weaken the condition in Theorem~\ref{thm:optimalsample} with the number of observations satisfying $n  > \frac{c'k+\log(k(p-k))}{c^2\beta_{\min}^2 \rho_{2k}^2/4}$ for a constant $c'$ that controls the maximum number of inactive with inactive swaps.

\subsection{Example}
\label{subsec:example}

\sloppy

In this section, we present a simple example that highlights the advantages of using $\SW$ to handle correlated measurement matrices.  The particular example we use is motivated from \citet{javanmard2013model}.  Recall the linear model (\ref{eq:linmodel}) and suppose that $\Sigma = X^T X/n$ is given as follows:
\begin{equation}
\Sigma_{i,j} = \left\{
\begin{array}{ll}
1 & i = j \\
a & i = p, \, j \in S^* \text{ or } j = p, \, i \in S^* \\
0 & \text{otherwise}
\end{array} \,,
\right. \label{eq:example}
\end{equation}
where $\Sigma_{i,j}$ refers to the $(i,j)$ entry of the matrix $\Sigma$, $S^* = \{1,2,\ldots,k\}$ is the unknown support that we seek to estimate given $y$ and $X$, and $a$ is chosen such that $\Sigma$ is positive definite.  From (\ref{eq:example}), we see that the $p^{\text{th}}$ column of $X$ is correlated with the columns indexed by the true support.  This correlation causes standard sparse linear regression algorithms to estimate the incorrect support.  For example, a simple calculation shows that the Lasso can only estimate $S^*$ accurately when $a \in [0, 1/k)$.  The calculations in \citet{javanmard2013model} show that multi-stage methods can estimate $S^*$ accurately when $a \in [0,1/k) \cup (1/k,1/\sqrt{k}]$.

\fussy

To find the range of values of $a$ for which $\SW$ can accurately estimate $S^*$, we need to compute the parameter $\gamma_k$ defined in (\ref{eq:gammaD}).  In particular, as shown in Theorem~\ref{thm:EqualCaseTwo}, when $\sigma = $, $\SW$ can perform accurate support recovery as long $\gamma_k < 1$.  A simple calculation shows that $\gamma_k = 0$.  This means that $\SW$ can estimate the true support as long $a \ge 0$ and $a$ is chosen such that $\Sigma$ is positive definite.  

The above example also demonstrates how $\SW$ can be used with other sparse regression algorithms.  In particular, if $a \in [0,1/k)$, then Lasso outputs the true support.  In this case, $\SW$ will also return the true support.  However, when $a > 1/k$, then Lasso is no longer accurate.  In this case, $\SW$ can be used with the output of the Lasso algorithm to estimate the true support.  As we shall see in the numerical simulations, one of the main advantages of using $\SW$ with other sparse regression algorithms is that the number of iterations needed for $\SW$ to converge can be significantly reduced.

\subsection{Discussion}
\label{subsec:discussion}

In this section, we make some general remarks regarding the theoretical results and discuss some possible extensions of our results.

\begin{remark}
\textbf{(Unknown Sparsity Level)}
In all the results, we assumed that $k$ is known.  However, in practice, $k$ is typically unknown and must be selected using an appropriate model selection algorithm.  The theoretical analysis for the case when $k$ is not known is not within the scope of this paper.  However, we note that recent work has shown that, under appropriate conditions, $k$ can be estimated exactly with high probability by computing the solution path of any high-dimensional consistent sparse regression algorithm \citep{VatsBaraniuk2014}.
\end{remark}

\begin{remark}
\textbf{(Number of $\SW$ Iterations)}
Our theoretical results specified sufficient conditions on the scaling of the number of observations and sufficient conditions on the measurement matrix for accurate support recovery.  An open question is to study the number of iterations required by $\SW$ to converge.
\end{remark}

\sloppy 
\begin{remark}
\textbf{(Random Measurement Matrices)}
In all the results, we assumed that the measurement matrix is deterministic.  
An interesting extension will be to study the case when the rows of $X$ are sampled from a random vector.  Such an analysis will have applications in compressed sensing \citep{donoho2006compressed} and Gaussian graphical model selection \citep{MeinshausenBuhl2006}.
\end{remark}

\fussy

\begin{remark}
\textbf{(High Correlations)}
Our results do not explicitly take into account the correlations among the columns of $X$, but simply improve upon the conditions required by other sparse regression algorithms.  We note that in settings where exact support recovery is not feasible, we cannot expect $\SW$ to output the true support.  This may happen, for example, when the conditions in Theorem~\ref{thm:ESDFirst} are not satisfied.  In such cases, it may be desirable to select a superset of the true support instead of the true support.  Our future work will study how $\SW$ can be used for this problem by using $\SW$ in conjunction with algorithms, such as those proposed in \citet{buhlmann2012correlated}, to select an appropriate superset of the true support.
%One way of selecting a superset is using the methods in \citet{} that clusters columns that are correlated to each other and then applies a group Lasso estimator \citep{} to select the superset.  In future work, we
\end{remark}

\section{Numerical Simulations}
\label{sec:sims}

In this section, we illustrate the performance of $\SW$ when initialized by several popular sparse regression algorithms.  Section~\ref{subsec:syn} presents synthetic data results and Section~\ref{subsec:real} presents pseudo real data results.

\subsection{Synthetic Data}
\label{subsec:syn}

To illustrate the advantages of $\SW$, we use the following two examples:

\begin{enumerate}[({E}1)]
\item We sample the rows of $X$ from a Gaussian distribution with mean zero and covariance $\Sigma$.  The covariance $\Sigma$ is block-diagonal with blocks of size $p/k$.  The entries in each block $\bar{\Sigma}$ are specified as follows: 
$\bar{\Sigma}_{ii} = 1$ for $i \in \{1,\ldots,p/k\}$ and $\bar{\Sigma}_{ij} = a$ for $i \ne j$.  This construction of the measurement matrix is motivated from \citet{buhlmann2012correlated}.  The true support is chosen such that each variable in the support is assigned to a different block.  The non-zero entries in $\beta^*$ are chosen to have magnitude one with sign randomly chosen to be either positive or negative (with equal probability).  We let $\sigma = 1$, $p = 1000$, $n \in \{100, 150,200,250,\ldots,500\}$, $k = 20$, and 
$a \in [0.6,0.99]$.
\item We sample $X$ from the same distribution as described in (E1).  The only difference is that the true support is chosen such that five different blocks contain active variables and each chosen block contains \textit{four} active variables.  The rest of the parameters are also the same.
\end{enumerate}
In both (E1) and (E2), as $a$ increases, the strength of correlations between the columns increase.  Furthermore, from the construction, it is clear that the restricted eigenvalue parameter for (E1) is, in general, greater than the restricted eigenvalue parameter of (E2).  Thus, (E1) requires a smaller number of observations than (E2) for accurate sparse regression.

\subsubsection{Sparse Regression Algorithms}
We use the following sparse regression algorithms to initialize $\SW$: 

\begin{itemize}

\item Lasso \citep{TibshiraniLasso1994}

\item  Thresholded Lasso (TLasso) \citep{GeerBuhlZhouAdaptive2011}

\item Forward-Backward (FoBa) \citep{zhang2011adaptive}

\item  CoSaMP \citep{needell2009cosamp}

\item  Marginal Regression (MaR)

\item  Random

\end{itemize}

TLasso is a two-stage algorithm where the first stage applies Lasso and the second stage selects the top $k$ largest (in magnitude) variables from the Lasso estimate.  In our implementation, we applied Lasso using $5$-fold cross-validation.  FoBa uses a combination of a forward and a backwards algorithm.  CoSaMP is an iterative greedy algorithm.  MaR selects the support by choosing the largest $k$ variables in $|X^T y|$.  Random selects a \textit{random} subset of size $k$.  We use the notation S-TLasso to refer to the algorithm that uses TLasso as an initialization for $\SW$.  A similar notation follows for other algorithms.  Finally, when running the sparse regression algorithms, we assume that $k$ is known.  In practice, model selection algorithms can be used to select an appropriate $k$.  Since our main goal is to illustrate the performance of $\SW$, and thereby validate our theoretical results in Section~\ref{sec:theory}, we only show results for the case when $k$ is assumed to be known.  Finally, all our results are reported over $100$ trials.

\subsubsection{Dependence on the Degree of Correlation}
Figures~\ref{fig:tpr_E1_Corr_Plot} and \ref{fig:tpr_E2_Corr_Plot} plot the mean TPR for (E1) and (E2), respectively, as the parameter $a$ increases from $0.5$ to $0.99$.  The dashed lines correspond to a standard sparse regression algorithm, while the solid lines correspond to a $\SW$ based algorithm.  In most cases, we see that, for the same color, the solid lines lie above the dashed lines.  This shows, as predicted by our theory, that $\SW$ is able to improve the performance of other regression algorithms.  Furthermore, we observe that TLasso has the best performance among all algorithms, and S-TLasso improves the performance of TLasso.  However, the computational complexity of TLasso is higher than that of other algorithms since it requires two stages of model selection.

As expected, the performance of the algorithms degrades as the correlations increase.  Moreover, when the correlations are extremely high, then the difference between TLasso and S-TLasso is insignficant.  This suggests that for such cases, it might be more suitable to use sparse regression algorithms, such as those in \citet{zou2005regularization, she2010,grave2011trace,huang2011sparse,buhlmann2012correlated}, which are designed to estimate a superset of the true support.

Figures~\ref{fig:diff_tpr_E1_Corr_Plot} and \ref{fig:diff_tpr_E2_Corr_Plot} show boxplots of the difference between the TPR of a $\SW$ based algorithm minus the TPR of a standard algorithm.  A positive difference means that $\SW$ is able to improve the performance of a sparse regression algorithm.  For Lasso, FoBa, and CoSaMP, we see that the difference is generally positive, even for large values of $a$.  For TLasso, we see that the difference is generally positive for $a = 0.90$ and $a = 0.93$.  For greater values of $a$, there does not seem to be any advantages of using $\SW$.  This is likely due to the correlations being extremely high so that the true support no longer minimizes the loss function.

\subsubsection{Number of Iterations and Dependence on the Number of Observations}
Figure~\ref{fig:numiter_E1E2} plots the mean number of iterations required by the $\SW$ based algorithms as the parameter $a$ varies from $0.55$ to $0.99$.  As expected, the number of iterations generally increases as the correlations among the columns increase.  We also observe that FoBa and TLasso require the smallest number of iterations to converge.  This is primarily because both these algorithms are able to estimate a large fraction of the true support.

Figure~\ref{fig:numN_E1E2} plots the mean TPR for (E1) and (E2) when $a = 0.9$ and the number of observations vary from $100$ to $500$.  We clearly see that $\SW$ based algorithms outperform the standard algorithms.  For simplicity, we only plot results for TLasso and FoBa.  %Other algorithms had similar interpretations in results.

\begin{figure}
\begin{center}
\includegraphics[scale=0.4]{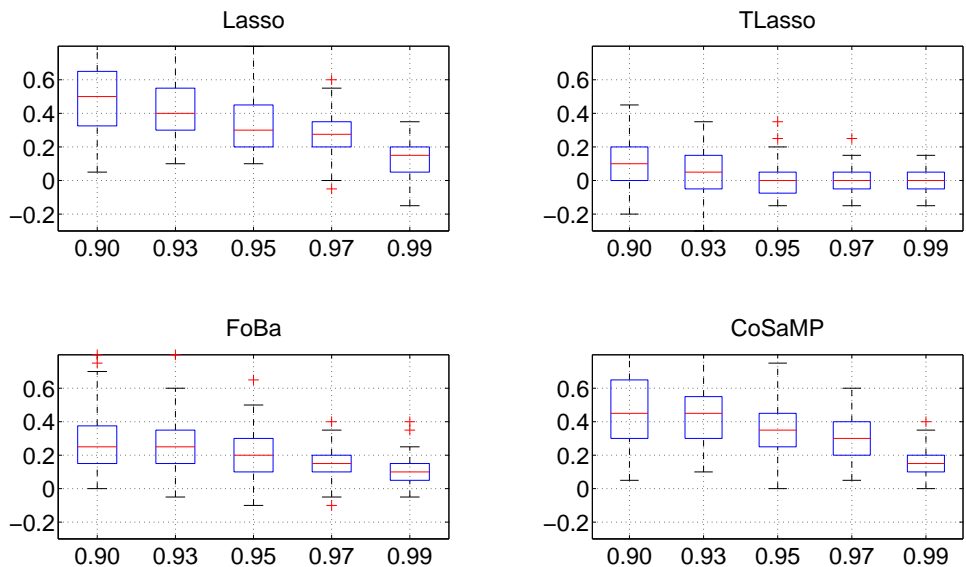}
\end{center}
\caption{Results for measurement matrix (E1).  Box plots of the TPR of $\SW$ based algorithms minus the TPR of a standard regression algorithm for five different values of the degree of correlation $a$.  For example, top left is the TPR of $\SW$ based Lasso minus the TPR of Lasso.}
\label{fig:diff_tpr_E1_Corr_Plot}
\end{figure}

\begin{figure}
\begin{center}
\vspace{-0.2cm}
\includegraphics[scale=0.8]{Legend_New.eps}
\vspace{-0.7cm}
\end{center}
\begin{center}
\includegraphics[scale=0.46]{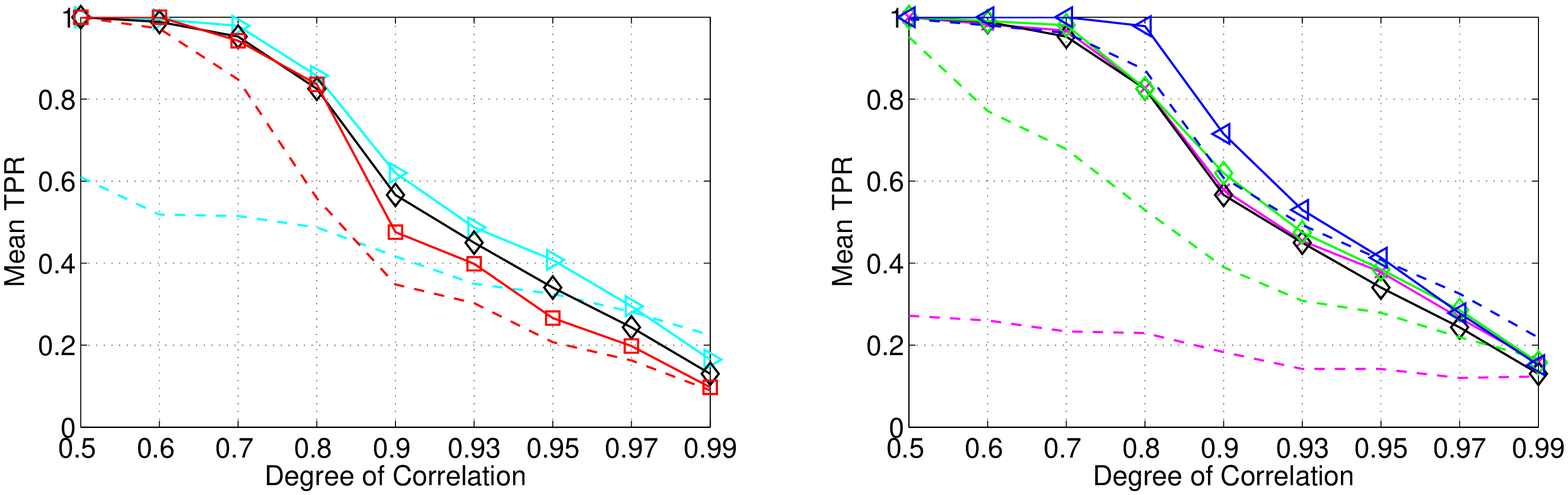}
\end{center}
\caption{ Results for measurement matrix  (E2). Mean true positive rate (TPR) versus the degree of correlation (the parameter $a$) for several different sparse regression algorithms.  The dashed lines correspond to standard sparse regression algorithms, while the solid lines with markers correspond to $\SW$ based regression algorithms. }
\label{fig:tpr_E2_Corr_Plot}
\begin{center}
\includegraphics[scale=0.4]{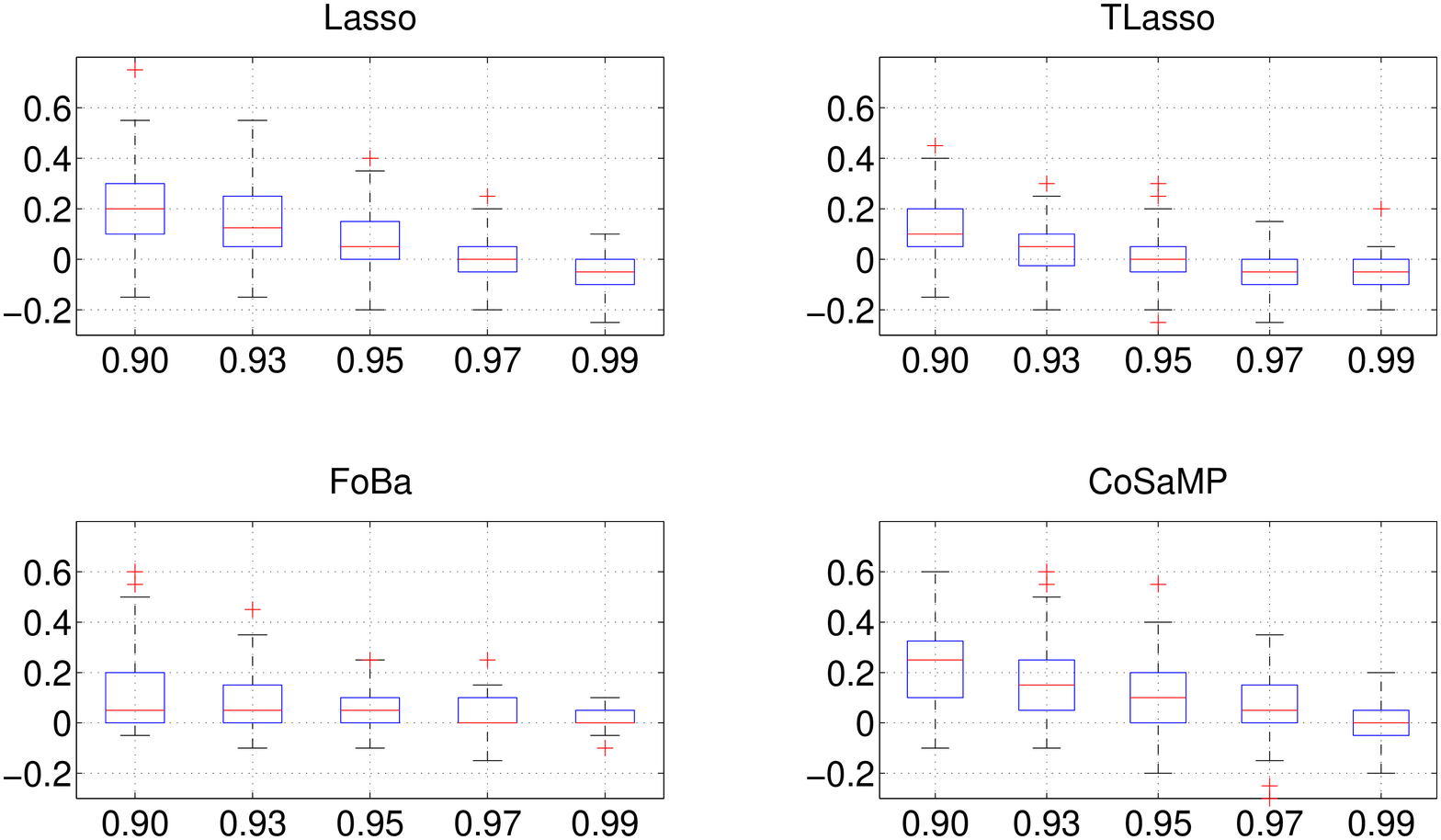}
\end{center}
\caption{Results for measurement matrix (E2). Box plots of the TPR of $\SW$ based algorithms minus the TPR of a standard regression algorithm for five different values of the degree of correlation $a$.  For example, top left is the TPR of $\SW$ based Lasso minus the TPR of Lasso.}
\label{fig:diff_tpr_E2_Corr_Plot}
\end{figure}

\begin{figure}
\begin{center}
\includegraphics[scale=0.46]{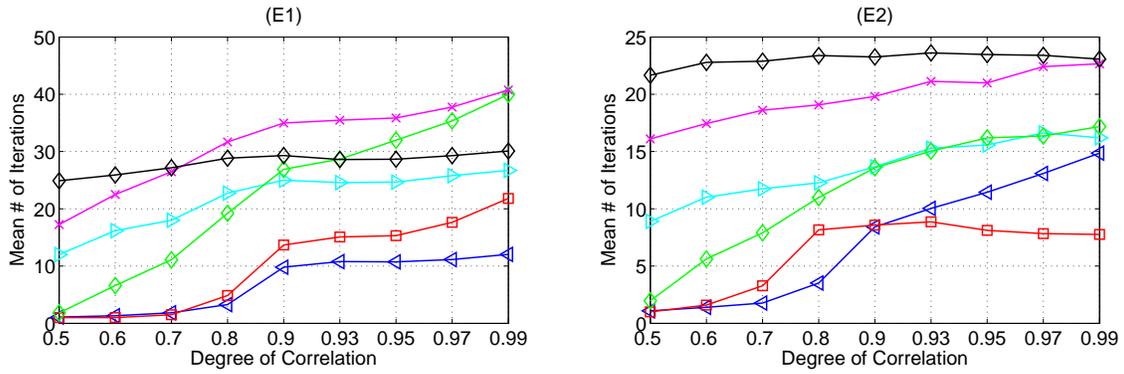}
\caption{Mean number of iterations required by $\SW$ when used with different sparse regression algorithms for the synthetic examples (E1) and (E2).  Recall that $p = 1000$, $n = 200$, and $k = 20$.  See Figure~\ref{fig:tpr_E2_Corr_Plot} for legend.}
\label{fig:numiter_E1E2}
\end{center}
\end{figure}

\begin{figure}
\begin{center}
\includegraphics[scale=0.4]{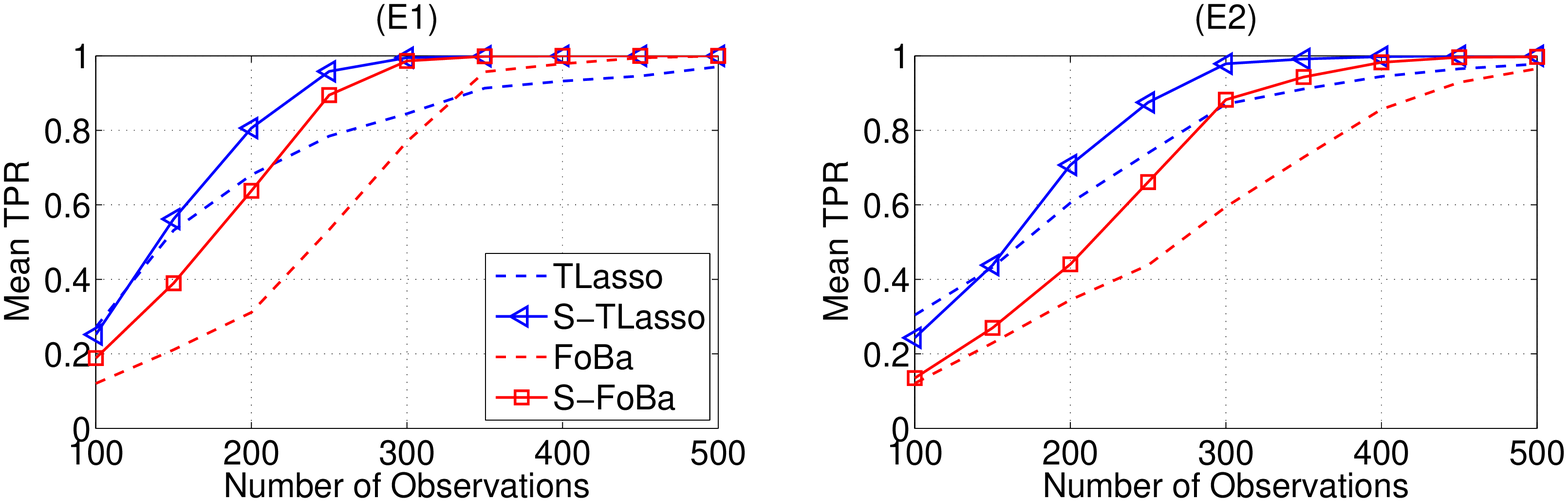}
\caption{Mean TPR versus the number of observations $n$ for (E1) and (E2).}
\label{fig:numN_E1E2}
\end{center}
\end{figure}

\subsection{Pseudo Real Data}
\label{subsec:real}

\sloppy 
In this section, we present results for the case when the measurement matrix $X$ corresponds to gene expression data.  Since we do not have any ground truth, we simulate $\beta^*$ and $y$.  For this reason, we refer to this simulation as pseudo real data.  The two simulation settings are as follows:
\begin{itemize}
\item Leukemia: This dataset contains $5147$ gene expression values from $72$ patients \citep{golub1999molecular}.  For computational reasons, we select only $p = 2000$ genes to obtain a $72 \times 2000$ measurement matrix $X$.  We let $k = 10$, $\beta_{\min} = 4$, and $\sigma = 1$.  As seen in Section~\ref{subsec:syn},
the choice of the support plays an important role in determining the performance of a sparse regression algorithm.  To select the support, we cluster the columns using kmeans to identify $20$ clusters.  Next, we obtain the support by random selecting five clusters and then selecting exactly two variabes from each cluster,

\item Prostate Cancer: This dataset contains $12533$ gene expression values from $102$ patients \citep{singh2002gene}.  The selection of $X$, $\sigma$, and $\beta_{\min}$ is the same as in the Leukemia data.  The only difference is that $k = 15$ and the selected support contains three variables from one cluster (as opposed to the two chosen in the Leukemia data).
\end{itemize}

\fussy 

Figure~\ref{fig:solutionPath} plots the mean TPR over $100$ realizations versus the sparsity level of the estimated support for several sparse regression algorithms.  In the figures, we also compare to the elastic net (ENET) algorithm \citep{zou2005regularization}, which is known to be suitable for regression with correlated variables.  Note that ENET requires two regularization parameters.  In our simulations, we run ENET for a two-dimensional grid of regularization parameters and select a support for each sparsity level that results in the smallest loss.  We only compare ENET to TLasso and FoBa, since we know from Section~\ref{subsec:syn} that these algorithms are superior to other regression algorithms in this situation.  From the figures, it is clear that, after a certain sparsity level, the $\SW$-based algorithms perform better than non-$\SW$-based algorithms.  Furthermore, we clearly see that $\SW$ performs significantly better than ENET.

\begin{figure}
\begin{center}
\includegraphics[scale=0.4]{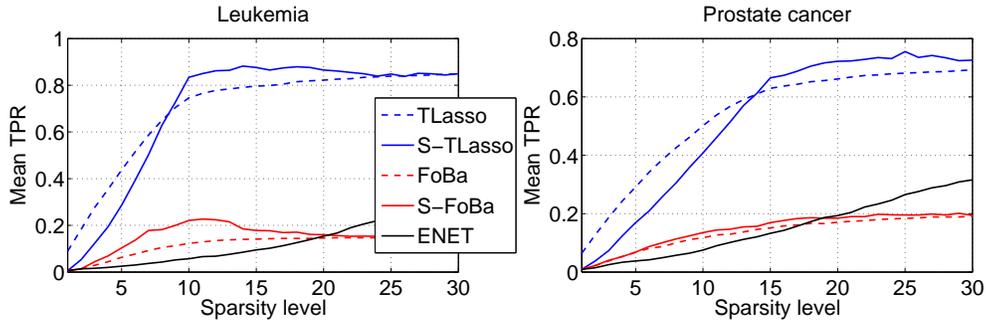}
\caption{Mean TPR as the sparsity level of the estimated support increases for the Leukemia gene expression data (left) and the prostate cancer gene expression data (right).  For the leukemia data, the true sparsity level is $k = 10$.  For the prostate cancer data, the true sparsity level is $k = 15$.}
\label{fig:solutionPath}
\end{center}
\end{figure}

\section{Conclusions}
\label{sec:summary}

Sparse regression is an important tool for analyzing a wide array of high-dimensional datasets.  However, standard computationally tractable algorithms can output erroneous results in the presence of correlated measurements.  In this paper, we developed and analyzed a simple sparse regression algorithm, called $\SW$, that can be used with any regression algorithm to boost its performance.  The main idea behind $\SW$ is to iteratively swap variables, starting with an initial estimate of the support, until a loss function cannot be reduced any further.  We have theoretically justified the use of $\SW$ with other regression algorithms and quantified the conditions on the measurements that guarantee accurate support recovery.  Using numerical simulations on synthetic and real gene expression data, we have shown how $\SW$ boosted the performance of several state-of-the-art sparse regression algorithms.  Our work in this paper motivates several interesting extensions of $\SW$, some of which were discussed in Section~\ref{subsec:discussion}.  Two interesting extensions of $\SW$, not discussed in Section~\ref{subsec:discussion}, are to extend $\SW$ for logistic regression and to extend $\SW$ for estimating structured sparse vectors.

\section*{Acknowledgements}

Thanks to Aswin Sankaranarayanan, Christoph Studer, and Eric Chi for valuable feedback and discussions.  This work was supported by an Institute for Mathematics and Applications (IMA) Postdoctoral Fellowship and by the Grants NSF IIS-1124535, CCF-0926127, CCF-1117939; ONR N00014-11-1-0714, N00014-10-1-0989; and ARO MURI W911NF-09-1-0383.

\appendix

%\section{Proof of Proposition~\ref{thm:basicl2}}
%\label{app:basicl2}

\section{Proof of Proposition~\ref{thm:ESDFirst}}
\label{app:ESDFirst}

Recall that ${\cal L}(S;y,X) = \| \Pi^{\perp}[S]y\|_2^2$.  Analyzing the exhaustive search decoder (ESD) for $s = k$ is equivalent to finding conditions under which the following holds:
\begin{equation}
|| \Pi^{\perp}[S^*] y ||_2^2 < \min_{S \in \Omega_k \backslash S^*} || \Pi^{\perp}[S] y ||_2^2 \,. \label{eq:main}
\end{equation}
Using the properties of the orthogonal projection, it is easy to see that $|| \Pi^{\perp}[S^*] y ||_2^2 = || \Pi^{\perp}[S^*] w ||_2^2$.  Furthermore, we have
\[
|| \Pi^{\perp}[S] y ||_2^2 = \xi^T \Pi^{\perp} [S] \xi + w^T \Pi^{\perp}[S] w + 2 \xi^T \Pi^{\perp}[S] w \,,
\]
where $\xi = X \beta^*$.  
Substituting the above into (\ref{eq:main}), we have that (\ref{eq:main}) is equivalent to showing that the following holds:
\[
 \min_{S \in \Omega_k \backslash S^*}  \left[
\underbrace{\xi^T \Pi^{\perp}[S] \xi + w^T ( \Pi^{\perp}[S] - \Pi^{\perp}[S^*]) w + 2 \xi^T \Pi^{\perp}[S] w}_{{\cal W}(S)}
\right] > 0 \,.
\]
To find conditions under which the above holds, 
we first lower bound ${\cal W}(S)$.  Using properties of projection matrices, and using arguments in \citet{wainwright2009information}, $\Pi^{\perp}[S] - \Pi^{\perp}[S^*]$ is a difference of two rank $\ell = |S^* \backslash S| = |S \backslash S^*|$ projection matrices.  Using properties of sub-Gaussian random vectors in Lemma~\ref{lemma:secondSG} and Lemma~\ref{lemma:thirdSG}, we have
\[
\P\left(
|w^T ( \Pi^{\perp}[S] - \Pi^{\perp}[S^*]) w| > 4 f_n \ell \sigma^2
\right) \le 2 e^{-\ell f_n/2} \,, \quad f_n \ge 1 \,,
\]
\[
\P\left( | \xi^T \Pi^{\perp}[S] w| \ge \delta_n \right) \le 2 e^{-\delta_n^2 / (2 || \xi^T \Pi^{\perp}[S] ||_2^2 \sigma^2)} \,.
\]
Using the above tail inequalities, we can write down a lower bound for ${\cal W}(s)$ such that
\[ {\cal W}(s) \ge || \Pi^{\perp}[S] \xi||_2^2 \left[
1 - \frac{4 \sigma^2 \ell f_n}{||\Pi^{\perp}[S] \xi||_2^2}
- \frac{ 2\sigma \delta_n}{||\Pi^{\perp}[S] \xi||_2}
\right] \,, \]
holds with probability at least $1 - 2 e^{-\ell f_n/2} - 2 e^{-\delta_n^2/2}$.  Using properties of the eigenvalues, we know that $|| \Pi^{\perp}[S] \xi ||_2^2 \ge \ell n\rho_{2k} \beta_{\min}^2$, where $\rho_{2k}$ is defined in \eqref{eq:rhokl}.  Thus, we have the following lower bound:
\[
{\cal W}(s) \ge  || \Pi^{\perp}[S] \xi||_2^2 \left[
1 - \frac{4 \sigma^2  f_n}{n \rho_{2k} \beta_{\min}^2}
- \frac{ 2 \sigma \delta_n}{\sqrt{n \ell \rho_{2k}} \beta_{\min}}
\right] \,, \]
which holds with probability at least $1 - 2 e^{-\ell f_n/2} - 2 e^{-\delta_n^2/2} $.
Choosing $f_n = \delta_n^2/\ell$,
\[
{\cal W}(s) \ge  || \Pi^{\perp}[S] \xi||_2^2 \left[
1 - \frac{2 \sigma^2 \delta_n^2}{n \ell \rho_{2k} \beta_{\min}^2}
- \frac{ 2\sigma \delta_n}{\sqrt{n \ell \rho_{2k}} \beta_{\min}}
\right]   \,, \]
with probability at least $1 - 4 e^{-\delta_n^2/2}$.
Choosing $\delta_n^2 = c^2 n \ell \rho_{2k} \beta_{\min}^2$,
\[
{\cal W}(s) \ge  || \Pi^{\perp}[S] \xi||_2^2 \left(
1 - 2 c^2 \sigma^2 -2 c \sigma
\right) \,. \]
with probability at least $1 - 4 e^{ - c^2 n \ell \rho_{2k} \beta_{\min}^2/2 }$.
Now, if $c$ is chosen such that $(1 - 2 (c\sigma)^2 - 2c\sigma) > 0$, then
\begin{equation}
\P( {\cal W}(S) < 0) \le 4 e^{-c^2 n \ell \rho_{2k} \beta_{\min}^2 / 2}  \,. \label{eq:ww}
\end{equation}
Recall that we want to find conditions under which $\min_{S \in \Omega_k \backslash S^*} {\cal W}(S) > 0$.  Using standard arguments from probability theory, we have
\begin{eqnarray*}
\P\left( \min_{S \in \Omega_k \backslash S^*} {\cal W}(S) > 0 \right) &=&
\P\left( \bigcap_{S \in \Omega_k \backslash S^*} \{ {\cal W}(S) > 0\} \right) \,,
= 1 - \P\left( \bigcup_{S \in \Omega_k \backslash S^*} \{ {\cal W}(S) \le 0\}\right) \\
&\ge& 1 - \sum_{S \in \Omega_k \backslash S^*} \P({\cal W}(S) \le 0) \,.
\end{eqnarray*}
Let $N(\ell)$ be the number of supports of size $k$ that differ from $S^*$ by $\ell$ variables.  From \citet{wainwright2009information}, we have $N(\ell) = \binom{k}{\ell} \binom{p-k}{\ell}$.  The summation above can now be upper-bounded as follows:
\begin{eqnarray*}
 \sum_{S \in \Omega_k \backslash S^*} \P({\cal W}(S) \le 0) &\le&
 \sum_{\ell=1}^{k} 4 N(\ell) \exp(-c^2 n \ell \rho_{2k} \beta_{\min}^2 / 2) \\
 &\le &  \sum_{\ell=1}^{k} 4 \binom{k}{\ell} \binom{p-k}{\ell}\exp(-c^2 n \ell \rho_{2k} \beta_{\min}^2 / 2) \,, \\
 &\le & 4k \max_{\ell=1,\ldots,k}  \binom{k}{\ell} \binom{p-k}{\ell}\exp(-c^2 n \ell \rho_{2k} \beta_{\min}^2 / 2) \,, \\
 &\overset{(a)}{\le} & 4k \max_{\ell=1,\ldots,k} \left( \frac{k(p-k) e^2}{\ell^2} \right)^{\ell}\exp(-c^2 n \ell \rho_{2k} \beta_{\min}^2 / 2) \,, \\
 &\le & 4k \max_{\ell=1,\ldots,k} \exp\left(
\ell \log(k(p-k)e^2/\ell^2) -c^2 n \ell \rho_{2k} \beta_{\min}^2 / 2
\right) \,, \\
&\le &  \max_{\ell=1,\ldots,k} \exp \left(
\ell ( \log(4k^2(p-k)e^2/\ell^2) - c^2 n \rho_{2k} \beta_{\min}^2 / 2
\right)) \,,
\end{eqnarray*}
where (a) uses the upper bound $\binom{p}{k} = (pe/k)^k$.  We want the above expression to diverge to $0$ as $(n,p,k) \rightarrow \infty$.  For this to happen, it is sufficient to require that
\[
n > \frac{4 + \log(k^2 (p-k)) }{ c^2 \beta_{\min}^2 \rho_{2k}/2  } \,,
\]
where the factor of $4$ comes from $4 > \log(4e^2)$; and recall that $c$ is chosen such that $1 - 2 (c \sigma)^2 - 2c\sigma > 0$.  It is easy to see that is sufficient to choose $c$ such that $0 < c \le 1/(3\sigma) \Longrightarrow 0 < c^2/2 \le 1/(18\sigma^2)$.

\section{Proof of Theorem~\ref{thm:EqualCaseTwo}}
\label{app:EqualCaseTwo}

Let $\Omega_k$ be the set of all supports of size $k$.  Suppose that for a set $S \in \Omega_k$, there exists a variable $i \in S$ that can be swapped with a variable $i' \in S^c$ such that the resulting support, $S^{(i,i')} = \{S \backslash i\} \cup \{i'\}$, leads to a lower loss than the loss associated with the support $S$.  In this case, $\SW$ will not stop and find a suitable swapping to reduce the loss.  If every set $S \in \Omega_k$, except $S^*$, has this property, then $\SW$ will only stop once it reaches~$S^*$.  Thus, a sufficient condition for $\SW$ to output the correct support is for the following two conditions to hold:
\begin{align}
\forall \; S \in \Omega_k,\quad \min_{i \in S \cap (S^*)^c, i' \in S^c \cap S^*} {\cal L}(S^{(i,i')}; y, X) &< {\cal L}(S;y,X) \,,\label{eq:condone}\\
\min_{S \in \Omega_k \backslash S^*} {\cal L}(S;y,X) &> {\cal L}(S^*;y,X) \,. \label{eq:condtwo}
\end{align}
Equation~\eqref{eq:condone} ensures that $\SW$ does not stop for all $S \in \Omega_k \backslash S^*$, and \eqref{eq:condtwo} ensures that $\SW$ stops once the true support has been identified.  
Note that \eqref{eq:condone} ensures that an inactive variable from $S$ may be swapped with an active variable from $S^c$.  
Define the events 
\begin{align*}
{\cal E} &= \{{\cal L}(S;y,X) > {\cal L}(S^*;y,X)\} \,, \\
{\cal D}_S &= \left\{ \min_{i \in S \cap (S^*)^c, i' \in S^c \cap S^*} {\cal L}(S^{(i,i')}; y, X) < {\cal L}(S;y,X) \right\},
\end{align*} 
where $S \in \Omega_{k}$.  The probability of accurate support recovery can be lower bounded as follows:
\begin{align}
\Pr(\widehat{S} = S^*) &\ge \Pr\left(\widehat{S} = S^* | \left\{ \cap_{S \in \Omega_k \backslash S^*} {\cal D}_{S} \right\} \cap {\cal E}\right) \Pr\left( \left\{ \cap_{S \in \Omega_k \backslash S^*} {\cal D}_{S} \right\} \cap {\cal E}\right) \nonumber \,, \\
&= \Pr\left( \left\{ \cap_{S \in \Omega_k \backslash S^*} {\cal D}_{S} \right\} \cap {\cal E}\right) \nonumber \,, \\
&= 1 - \Pr\left( \cup_{S \in \Omega_k\backslash S^*} {\cal D}_{S}^c \right) - \Pr({\cal E}^c) \,.
\label{eq:ttt12}
\end{align}
Theorem~\ref{thm:ESDFirst} identifies conditions under which $\Pr({\cal E}^c) \rightarrow 0$.  Thus, we only need to specify conditions under which $\Pr\left( \cup_{S \in \Omega_k\backslash S^*} {\cal D}_{S}^c \right) \rightarrow 0$.  To do so, we first analyze the event ${\cal D}_S^c$ for a fixed $S$.  Using the definition of the least-squares loss, we have
\begin{align}
{\cal D}_S^c &= \left\{ \min_{i \in S \cap (S^*)^c, i' \in S^c \cap S^*} \| \Pi^{\perp}[S^{(i,i')}] y \|_2^2 - \| \Pi^{\perp}[S] y \|_2^2 > 0 \right\}  \,, \nonumber \\
&= \left\{ \min_{i \in S \cap (S^*)^c, i' \in S^c \cap S^*} \left[ \xi^T \Gamma_S(i,i') \xi + w^T \Gamma_S(i,i') w + 2 \xi^T \Gamma_S(i,i')w \right] > 0 \right\} \label{eq:simplea}\,,
\end{align}
%To analyze $\Pr({\cal D}_S^c)$, we first simplify the expression $\| \Pi^{\perp}[S^{(i,i')}] y \|_2^2 - \| \Pi^{\perp}[S] y \|_2^2$:
%\begin{align*}
%\| \Pi^{\perp}[S^{(i,i')}] y \|_2^2 - \| \Pi^{\perp}[S] y \|_2^2 &=
%\| \Pi[S] y \|_2^2 - \| \Pi[S^{(i,i')}] y \|_2^2 \\
%&= \xi^T \Gamma_S(i,i') \xi + w^T \Gamma_S(i,i') w + 2 \xi^T \Gamma_S(i,i') w \,,
%\end{align*}
where $\xi = X_{\bar{S}}\beta^*_{\bar{S}}$, $\bar{S} = S^* \backslash S$, $\Gamma_S(i,i') = \Pi[S] - \Pi[A_i,i']$, and $A_i = S \backslash \{i\}$.  Recall that $i \in S^c \cap (S^*)^c$, which leads to the expression in \eqref{eq:simplea}.  Note that the first term in the expression of (\ref{eq:simplea}) is deterministic, while the second and third terms are random (they depend on the noise $w$).
To show that $\Pr({\cal D}_S^c) \rightarrow 0$, we first upper bound $\Pr({\cal D}_S^c)$ and then show that the upper bound converges to $0$.  Using properties of projection matrices, it is easy to see that $\Gamma_S(i,i')$ is a difference of two rank one projection matrices.  Using Lemma~\ref{lemma:secondSG} and Lemma~\ref{lemma:thirdSG}, we have the following tail bounds:
\begin{align}
\Pr\left(|w^T \Gamma_S(i,i') w| \ge 4 f_n^{(i,i')} \sigma^2\right) 
&\le 2 e^{-f_n^{(i,i')}/2},\; f_n^{(i,i') } \ge 1 \,, \label{eq:tail1} \\
\Pr\left( |\xi^T \Gamma_S(i,i') w | \ge 
\sigma \| \xi^T \Gamma_S(i,i')\|_2 \delta_n^{(i,i')}  \right) 
&\le 2 e^{-\left(\delta_n^{(i,i')}\right)^2/ 2 } \,. \label{eq:tail2}
\end{align}
\fussy
To simplify notation, define the events ${\cal E}_1 = \left\{|w^T \Gamma_S(i,i') w| < 4 f_n^{(i,i')} \sigma^2 \right\}$ and ${\cal E}_2 = \left\{ |\xi^T \Gamma_S(i,i') w | < \sigma \| \xi^T \Gamma_S(i,i')\|_2 \delta_n^{(i,i')} \right\}$.  
We emphasize that both ${\cal E}_1$ and ${\cal E}_2$ depend on $S$, $i$, and $i'$.  Using the total probability theorem, we have the following upper bound for $\Pr({\cal D}_S^c)$:
\sloppy
\begin{align}
\Pr({\cal D}_S^c) &\le \min_{i \in S \cap (S^*)^c, i' \in S^c \cap S^*} 
\Pr( \overbrace{\xi^T \Gamma_S(i,i') \xi + w^T \Gamma_S(i,i') w + 2 \xi^T \Gamma_S(i,i')w > 0}^{{{\cal H}}_S^{(i,i')}}) \label{eq:tempc} \,, \\
\Pr({{\cal H}}_S^{(i,i')})&\le \Pr({{\cal H}}_S^{(i,i')} | {\cal E}_1 \cap {\cal E}_2) \Pr({\cal E}_1 \cap {\cal E}_2) + \Pr({{\cal H}}_S^{(i,i')} | {\cal E}_1^c \cup {\cal E}_2^c) \Pr({\cal E}_1^c \cup {\cal E}_2^c) \,, \\
&\le \Pr({{\cal H}}_S^{(i,i')} | {\cal E}_1 \cap {\cal E}_2) + \Pr({\cal E}_1^c) + \Pr({\cal E}_2^c) \,. \label{eq:tempb}
\end{align}
Next, using the definition of ${\cal E}_1$ and ${\cal E}_2$, we have
\begin{align}
\Pr({{\cal H}}_S^{(i,i')} | {\cal E}_1 \cap {\cal E}_2)
\le  \mathbbm{1}_{\R^+}\left( \xi^T \Gamma_S \xi + 4 f_n^{(i,i')} \sigma^2 + 2\sigma \| \xi^T \Gamma_S(i,i')\|_2 \delta_n^{(i,i')} \right) \,, \label{eq:tempa}
\end{align}
where $\mathbbm{1}_{\R^+}(q) = 1$ if $q > 0$ and $\mathbbm{1}_{\R^+}(q) = 0$ if $q \le 0$.  Choosing $f_n^{(i,i')} = (\delta_n^{(i,i')})^2 = \delta_n^2$ (i.e., independent of $i$ and $i'$), and substituting \eqref{eq:tempa} and \eqref{eq:tempb} into \eqref{eq:tempc}, we have the following upper bound for $\Pr({\cal D}_S^c)$:
\begin{align}
\mathbbm{1}_{\R^+}\left( \min_{i,i'} \left(\xi^T \Gamma_S(i,i') \xi + 2 \delta_n^2 \sigma^2 + 2\sigma \| \xi^T \Gamma_S(i,i')\|_2 \delta_n \right) \right)  + 
4 e^{-\delta_n^2/ 2 } \label{eq:tempupper1} \,,
\end{align}
where $i \in S \cap (S^*)^c$, $i' \in S^c \cap S^*$, and we use the tail bounds in \eqref{eq:tail1} and \eqref{eq:tail2}.  In Section~\ref{subsec:appprofeq}, we show that
\begin{align}
\min_{i \in S \cap (S^*)^c,i' \in S^c \cap S^*} &\left(\xi^T \Gamma_S(i,i') \xi + 2\sigma \| \xi^T \Gamma_S(i,i')\|_2 \delta_n + 2\delta_n^2 \sigma^2 \right) \nonumber \\
&\le C \left[ 
(\gamma_k - 1) + \frac{2 \sigma \delta_n}{\sqrt{n}  \rho_{2k} \beta_{\min}} \left(\sqrt{\gamma_k } + 1/\sqrt{\rho_{k,1}}\right) + \frac{2 \delta_n^2 \sigma^2}{n \rho_{2k}^2 \beta_{\min}^2} \right] \,, \label{eq:uppertemp1}
\end{align}
where $C > 0$, $\gamma_k$ is defined in (\ref{eq:gammaD}), and $\rho_{k,1}$ is defined in (\ref{eq:rhokd}).  Choosing $\delta_n = c_1 \sqrt{n} \rho_{2k} \beta_{\min}$, and using the bound in (\ref{eq:uppertemp1}), we further upper bound (\ref{eq:tempupper1}) as
\[
\mathbbm{1}_{\R^+}\left( (\gamma_k - 1) + {2c_1 \sigma} \left(\sqrt{\gamma_k} + 1/\sqrt{\rho_{k,1}}\right) + 2(c_1\sigma)^2\right)  + 
4 e^{-c_1^2 n \beta_{\min}^2 \rho_{2k}^2 / 2 } \,.
\]
If $c_1$ is chosen such that $(\gamma_k - 1) + 2c_1 \sigma \left(\sqrt{\gamma_k} + 1/\sqrt{\rho_{k,1}}\right) + 2(c_1\sigma)^2 \le 0$, then we have 
\[
\Pr({\cal D}_S^c) \le 4 e^{-c_1^2 n \beta_{\min}^2 \rho_{2k}^2 / 2 } \,.
\]
Substituting the above into (\ref{eq:ttt12}) and using the union bound, we have
\[
\Pr(\widehat{S} = S^*) \ge 1 - \binom{p}{k} 4 e^{-c_1^2 n \beta_{\min}^2 \rho_{2k}^2 / 2 } - \Pr({\cal E}^c) \,.
\]
Using Theorem~\ref{thm:ESDFirst}, we know that if $n > \frac{4 + \log(k^2(p-k))}{c^2 \beta_{\min}^2 \rho_{2k}/2}$, where $0 < c^2/2 \le 1/(18 \sigma^2)$, then $\Pr({\cal E}^c) \rightarrow 0$.  From the above expression, it is clear that if $n > \max \left\{ \frac{\log \binom{p}{k}} {c_1^2 n \beta_{\min}^2 \rho_{2k}^2/2}, \frac{4 + \log(k^2(p-k))}{c^2 \beta_{\min}^2 \rho_{2k}/2}\right\}$, then $\Pr(\widehat{S} = S^*) \rightarrow 1$.  Choosing $c_1 = c$ and realizing that $\rho_{2k}^2 < \rho_{2k}$, we obtain the desired result.

\subsection{Proof of Equation~(\ref{eq:uppertemp1})}
\label{subsec:appprofeq}

Using Lemma~\ref{lemma:blockdecomp}, we have
\begin{align}
\Gamma_S(i,i') = \Pi[S] - \Pi[A_i,i'] = \underbrace{\frac{(\Pi^{\perp}[A_i] X_i) (\Pi^{\perp}[A_i] X_i)^T}{\|\Pi^{\perp}[A_i] X_i\|_2^2}}_{P_i} - \underbrace{\frac{(\Pi^{\perp}[A_i] X_{i'}) (\Pi^{\perp}[A_i] X_{i'})^T}{\|\Pi^{\perp}[A_i] X_{i'}\|_2^2}}_{P_{i'}} \,. \nonumber 
\end{align}
Next, we evaluate $\xi^T P_i \xi$:
\begin{align*}
\xi^T P_i \xi &= \frac{(X_i^T \Pi^{\perp}[A_i] \xi)^2}{\| \Pi^{\perp}[A_i] X_i\|_2^2} \\
&\overset{(a)}{=} \frac{(X_i^T \Pi[A_i \cup \bar{S}] \Pi^{\perp}[A_i] \xi)^2}{\| \Pi^{\perp}[A_i] X_i\|_2^2} \,, \\
&\overset{(b)}{=} \frac{(X_i^T X_{A_i\cup \bar{S}} (X_{A_i\cup \bar{S}}^T X_{A_i\cup \bar{S}})^{-1} X_{A_i\cup \bar{S}}^T \Pi^{\perp}[A_i] \xi)^2}{\| \Pi^{\perp}[A_i] X_i\|_2^2} \,, \\
&\overset{(c)}{=} \frac{\left(X_i^T X_{A_i\cup \bar{S}} (X_{A_i\cup \bar{S}}^T X_{A_i\cup \bar{S}})^{-1} \left[0_{n \times |A_i| } \; X_{\bar{S}}\right]^T \Pi^{\perp}[A_i] \xi\right)^2}{\| \Pi^{\perp}[A_i] X_i\|_2^2} \,.
\end{align*}
Recall that $\xi = X_{\bar{S}} \beta^*_{\bar{S}}$.  Step~(a) follows since $\Pi^{\perp}[A_i]\xi$ is in the span of $A_i \cup \bar{S}$.  Step~(b) uses the definition of the projection matrix $\Pi[A_i \cup \bar{S}]$.  Step~(c) uses the fact that $X_{A_i}$ is orthogonal to the projection matrix $\Pi^{\perp}[A_i]$.  The notation $0_{n \times |A_i| }$ refers to a matrix of size $|A_i| \times n$ with all zeros.  We now want to find an upper bound for the last expression above.  Using Holder's inequality, we have
\begin{align}
\xi^T P_i \xi &\le \frac{\left|\left|\left[
X_i^T X_{A_i\cup \bar{S}} (X_{A_i\cup \bar{S}}^T X_{A_i\cup \bar{S}})^{-1} 
\right]_{\bar{S}}\right|\right|_1^2 \cdot 
\| X_{\bar{S}}^T \Pi^{\perp}[A_i] \xi\|^2_{\infty}}{\| \Pi^{\perp}[A_i] X_i\|_2^2} \,,
\label{eq:upperboundxipi}
\end{align}
where the notation $\| [ v_{A_i \cup \bar{S}} ]_{\bar{S}}\|_p$ computes the norm over the variables in $\bar{S}$.  For notational simplicity, define $\gamma_{i,S}$ as follows:
\begin{equation}
\gamma_{i,S} = \frac{\sqrt{n} \left|\left|\left[
X_i^T X_{A_i\cup \bar{S}} (X_{A_i\cup \bar{S}}^T X_{A_i\cup \bar{S}})^{-1} 
\right]_{\bar{S}}\right|\right|_1^2}{ \| \Pi^{\perp}[A_i]X_i\|_2^2 } \,.
\label{eq:gammais}
\end{equation}
Recall that our overall goal is to evaluate $\xi^T \Gamma_S(i,i') \xi$ and $\|\Gamma_S(i,i') \xi\|_2$.  We first find an upper bound for $\|\Gamma_S(i,i')\|_2$:
\begin{align}
\|\Gamma_S(i,i') \xi\|_2 &\overset{(a)}{\le} \|P_i \xi \|_2 + \|P_{i'} \xi\|_2 
\overset{(b)}{\le} \frac{\sqrt{\gamma_{i,S}} \| X_{\bar{S}}^T \Pi^{\perp}[A_i] \xi\|_{\infty}}{\| \Pi^{\perp}[A_i] X_i\|_2} + \frac{|X_{i'}^T \Pi^{\perp}[A_i] \xi|}{ \sqrt{n}} \,, \nonumber \\
&\overset{(c)}{\le} \frac{\sqrt{\gamma_{i,S}} \| X_{\bar{S}}^T \Pi^{\perp}[A_i] \xi\|_{\infty}}{\sqrt{n}} + \frac{\|X_{\bar{S}}^T \Pi^{\perp}[A_i] \xi|^2\|_{\infty}}{\| \Pi^{\perp}[A_i] X_{i'}\|_2} \,, \nonumber \\
&\overset{(d)}{\le} \| X_{\bar{S}}^T \Pi^{\perp}[A_i] \xi\|_{\infty} \left(\frac{\sqrt{\gamma_{i,S}} }{\sqrt{n}} + \frac{1}{\| \Pi^{\perp}[A_i] X_{i'}\|_2}\right) \,, \nonumber \\
&\overset{(e)}{\le} \| X_{\bar{S}}^T \Pi^{\perp}[A_i] \xi\|_{\infty} \left(\frac{\sqrt{\gamma_{i,S}} }{\sqrt{n}} + \frac{1}{\sqrt{n\rho_k}}\right)   \label{eq:xil2} \,.
\end{align}
Step~(a) is due to the triangle inequality.  Step~(b) substitutes expressions for $P_i$ and $P_{i'}$.  Step~(c) follows because $i' \in \bar{S}$.  Step~(d) is simple algebra.  Finally, step~(e) uses the fact that $\|\Pi^{\perp}[A_i]X_{j''}\| \ge \sqrt{n \rho_{k,1}}$ for all ${j''} \in \bar{S}$.  We now bound the minimum value of $\xi^T \Gamma_{S}(i,i') \xi$ for all $i' \in \bar{S}$:
\begin{align}
\min_{i' \in \bar{S}} \xi^T \Gamma_S(i,i') \xi &= \min_{i' \in \bar{S}} \left[\xi^T P_i \xi - \xi^T P_{i'} \xi \right] \,, \nonumber \\
&\overset{(a)}{\le} \frac{\gamma_{i,S} \| X_{\bar{S}}^T \Pi^{\perp}[A_i] \xi\|_{\infty}^2}{n} - \max_{i' \in \bar{S}} \frac{(X_{i'} \Pi^{\perp}[A_i] \xi)^2}{\| \Pi^{\perp}[A_i] X_{i'}\|_2^2} \,, \nonumber \\
&\overset{(b)}{\le} \frac{\gamma_{i,S} \| X_{\bar{S}}^T \Pi^{\perp}[A_i] \xi\|_{\infty}^2}{n} 
- \frac{(X_{\bar{S}} \Pi^{\perp}[A_i] \xi)^2}{\| \Pi^{\perp}[A_i] X_{i'}\|_2^2} \,, \nonumber \\
&\overset{(c)}{\le} \| X_{\bar{S}}^T \Pi^{\perp}[A_i] \xi\|_{\infty}^2 \left[\frac{\gamma_{i,S}}{n} - \frac{1}{n} \right] \label{eq:xigammaxi} \,.
\end{align}
Step~(a) uses the upper bound in (\ref{eq:upperboundxipi}) and the definition of $P_{i'}$.  Step~(b) uses the definition of the $\ell_{\infty}$-norm.  Step~(c) uses some simple algebera and the bound $n \rho_{k} < \| \Pi^{\perp}[A_i] X_{j''} \|_2^2 \le n$ for any $j'' \notin A_i$.
We are now ready to find an upper bound for the expression of interest:

\medskip

\noindent
$\displaystyle{\min_{i \in (S^*)^c \cap S, i' \in \bar{S}} 
\left[ \xi^T \Gamma_{S}(i,i') \xi + 2\sigma \delta_n \| \Gamma_S(i,i') \|_2 + 2\delta_n^2\sigma^2 \right]}$
\begin{align*}
& \overset{(a)}{\le} \min_{i \in (S^*)^c \cap S}\left[
\overbrace{\| X_{\bar{S}}^T \Pi^{\perp}[A_i] \xi\|_{\infty}^2}^{R_i^2} \left[\frac{\gamma_{i,S}}{n} - \frac{1}{n} \right]\!+\!2\sigma \delta_n  \frac{\| X_{\bar{S}}^T \Pi^{\perp}[A_i] \xi\|_{\infty}}{\sqrt{n}} \left(\sqrt{\gamma_{i,S} }\!+\! \frac{1}{\sqrt{\rho_k}}\right) + 2 \delta_n^2 \sigma^2 \right]  \,,\\
& \overset{(b)}{\le} \min_{i \in (S^*)^c \cap S} R_i^2 \left[\left(\frac{\gamma_{i,S}}{n} - \frac{1}{n} \right) + \frac{2\sigma \delta_n}{\sqrt{n} R_i} \left(\sqrt{\gamma_{i,S} } + \frac{1}{\sqrt{\rho_k}}\right) + \frac{2 \delta_n^2 \sigma^2}{R_i^2}  \right] \,,\\
&\overset{(c)}{\le}  \min_{i \in (S^*)^c \cap S} R_{i}^2 \left[
\left(\frac{\gamma_{i,S}}{n} - \frac{1}{n} \right) + \frac{2\sigma \delta_n}{ n \sqrt{n}\rho_{2k} \beta_{\min}} \left(\sqrt{\gamma_{i,S} } + \frac{1}{\sqrt{\rho_k}}\right) + \frac{2 \delta_n^2 \sigma^2}{n^2 \rho_{2k}^2 \beta_{\min}^2} \right] \,,\\
& \overset{(d)}{\le}  C \min_{i \in (S^*)^c \cap S} \left[
\left({\gamma_{i,S}} - 1 \right) + \frac{2\sigma \delta_n}{\sqrt{n}\rho_{2k} \beta_{\min}} \left(\sqrt{\gamma_{i,S} } + \frac{1}{\sqrt{\rho_{k,1}}}\right) + \frac{2 \delta_n^2 \sigma^2}{n \rho_{2k}^2 \beta_{\min}^2} \right] \,,\\
&\overset{(e)}{\le}  C \left[ 
(\gamma_k - 1) + \frac{2\sigma \delta_n}{\sqrt{n} \rho_{2k} \beta_{\min}} \left(\sqrt{\gamma_k } + 1/\sqrt{\rho_{k,1}}\right) + \frac{2 \delta_n^2 \sigma^2}{n \rho_{2k}^2 \beta_{\min}^2} \right] \,.
\end{align*}
Step~(a) uses the bounds obtained in (\ref{eq:xil2}) and (\ref{eq:xigammaxi}) and introduces the notation $R_i$ for simplicity.  The $i'$ is captured in the $\ell_{\infty}$-norm term.  Step~(b) is simple algebra.  In step~(c), we make use of the bound $R_i^2 = \|X_{\bar{S}}^T \Pi^{\perp}[A_i] X_{\bar{S}} \beta^*_{\bar{S}} \|_{\infty}^2 \ge n^2 \rho_{2k}^2 \|\beta^*_{\bar{S}}\|_2^2/|\bar{S}| \ge n^2 \rho_{2k}^2 \beta_{\min}^2$.  In Step~(d), we factor out the terms that do not depend on $i$ and represent them by $C$, where $C > 0$.  In Step~(e), we use the fact that $\min_{i \in (S^*)^c \cap S} \gamma_{i,S} = \gamma_k$, where $\gamma_k$ is defined in (\ref{eq:gammaD}).  This equivalence can be easily established using the block-inversion formula.

%Now, in Assumption~2(b), we want the expression in (\ref{eq:mainfn1}) to be negative.  Since (\ref{eq:mainfn1u}) is an upper bound for (\ref{eq:mainfn1}), it is clear that if the following condition is true
%\begin{align}
%\min_{i \in (S^*)^c \cap S} \left[\left(\frac{\gamma_{i,S}^2}{n} - \frac{1}{n} \right) + \frac{\sigma f_n/\sqrt{n}}{n \rho_{k+d} \beta_{\min}} \left({\gamma_{i,S} } + \frac{1}{\sqrt{\rho_k}}\right) + \frac{\sigma^2 f_n^2}{n^2 \rho_{k+d}^2 \beta^2_{\min}} \right] < 0
%\end{align}
%then Assumption~2(b) holds. 

\section{Proof of Theorem~\ref{thm:optimalsample}}
\label{app:optimalsample}

Suppose that, after $r$ iterations, $\SW$ outputs $\widehat{S} = S^{(r)}$.  To ensure that $\widehat{S} = S^*$, we want to impose conditions so that with each iteration, $\SW$ takes positive steps towards the true support $S^*$.  Let $S^{(1)},S^{(2)},\ldots,S^{(r)}$ be the intermediate supports computed in each iteration of $\SW$.  In what follows, the conditions we impose will ensure that 
\begin{equation}
|S^* \backslash S^{(1)}| \ge |S^* \backslash S^{(2)}| \ge \cdots \ge |S^* \backslash S^{(r-1)}| > |S^* \backslash S^{(r)}| = 0 \,. \label{eq:conditionS}
\end{equation}
In other words, with each iteration, the number of active variables missed by $\SW$ will not increase and eventually decrease to $0$.  In order to ensure that (\ref{eq:conditionS}) holds and that $r$ is much smaller than $\binom{p}{d}$, we define the following events:
\begin{align}
{\cal E}_{S^*} &= \{ S^* = \arg \min_{S \in \Omega_k} {\cal L}(S;y,X)\} \,, \\
{\cal D}_{S} &= \left\{ \min_{i \in S \cap (S^*)^c, i' \in S^c \cap S^*} {\cal L}(S^{(i,i')}; y, X) < {\cal L}(S;y,X) \right\} \,, \\
{\cal F}_{S} &= \left\{ \min_{i \in S \cap (S^*)^c, i' \in S^c \cap S^*} {\cal L}(S^{(i,i')}; y, X) < \min_{j \in S , j' \in S^c \cap S^*} {\cal L}(S^{(j,j')}; y, X) \right\} \,.
\end{align}
The event ${\cal E}_{S^*}$ is the set of outcomes where $S^*$ minimizes the loss.  The event ${\cal D}_{S}$ is the set of outcomes where there exists at least one inactive variable in $S$ that can be swapped with an active variable from $S^c$.  Finally, the event ${\cal F}_{S}$ is the set of outcomes where an inactive variable in $S$ can only be swapped with an active variable from $S^c$ and an active variable from $S$ cannot be swapped with an inactive variable from $S^c$.  Before analyzing the events ${\cal E}_{S^*}$, ${\cal D}_{S}$, and ${\cal F}_{S}$, we establish an upper bound for the number of iterations in $\SW$.
\begin{lemma}
\label{lemma:upperR}
If for every $S \in \Omega_k$ such that $|S^* \backslash S| \le d$, $\Pr({\cal E}_{S^*}) = \Pr({\cal D}_{S}) = \Pr({\cal F}_{S}) = 1$, then $r$ can be upper bounded as follows:
\begin{align}
r \le \left\{
\begin{array}{cc}
d \binom{k}{d} & d \le \lceil k/2 \rceil \\
2^k & d > \lceil k/2 \rceil
\end{array}
\right. \,. \nonumber 
\end{align}
\label{eq:upperR}
\end{lemma}
\begin{proof}
If $d \le 1$, it is clear that $\Pr({\cal E}_{S^*}) = 1$ ensures that $r = 1$ if $S = S^*$ and $r = 2$ if $|S \backslash S^*| = 1$.  If $d > 1$, then the conditions $\Pr({\cal D}_{S}) = \Pr({\cal F}_{S}) = 1$ ensures that in each iteration, either an active variable is swapped with an active variable or an inactive variable is swapped with an active variable.  For each $d' \le d$, there are $\binom{k}{k-d'}$ possible supports once the inactive variables have been fixed.  This means that the maximum possible value of $r$ is $1+ \sum_{d' = 2}^{d} \binom{k}{k-d'} = 1+\sum_{d'=2}^{d} \binom{k}{d'}$.  The upper bound in (\ref{eq:upperR}) follows using standard upper bounds of the binomial coefficients.
\end{proof}
For notational convenience, define the event ${\mathscr{E} = {\cal E}_{S^*} \bigcap_{\ell=1}^{r-1} {\cal D}_{S^{(\ell)}} \bigcap_{\ell=1}^{r-1} {\cal F}_{S^{(\ell)}}}$.  We have the following lower bound for the probability of correctly recovering the true support:
\begin{align}
\Pr(\widehat{S} = S^*) &\ge \Pr( \widehat{S} = S^* | {\mathscr{E}}) \Pr({\mathscr{E}}) = \Pr({\mathscr{E}}) \,, \nonumber \\
&\ge 1 - \Pr({\cal E}_{S^*}^c) - \sum_{\ell=1}^{r-1} \Pr({\cal D}_{S^{(\ell)}}^c) - \sum_{\ell=1}^{r-1} \Pr({\cal F}_{S^{(\ell)}}^c) \,, \nonumber \\
&\ge 1 - \Pr({\cal E}_{S^*}^c) - r \max_{S \in \Omega_{k,d} \backslash S^*} \Pr({\cal D}_{S}^c) - r \max_{S \in \Omega_{k,d} \backslash S^*} \Pr({\cal F}_{S}^c) \,. \label{eq:tttt1}
\end{align}
From Theorem~\ref{thm:EqualCaseOne}, we know that if $n > \frac{4 + \log(k^2 (p-k)) }{ c^2 \beta_{\min}^2 \rho_{2k}/2 }$, where $0 < c^2 \le 1/(18 \sigma^2)$, then $\Pr({\cal E}_{S^*}^c) \rightarrow 0$.  Furthermore, from the proof of Theorem~\ref{thm:EqualCaseTwo} in Appendix~\ref{app:EqualCaseTwo}, we know that $\Pr({\cal D}_S^c) \le 4 e^{-c^2 n \beta_{\min}^2 \rho_{2k}^2/2}$.  Thus, we only need to analyze $\Pr({\cal F}_S^c)$.  Using the definition of the least-squares loss, we have
\[
{\cal F}_S^c = \left\{ \min_{i,i'} 
\max_{j,j'} \left[ 
\xi^T \Theta_S(i,i',j,j') \xi + w^T \Theta_S(i,i',j,j') w + 2 \xi^T \Theta_S(i,i',j,j') w
\right] \ge 0 \right\} \,,
\]
where $\xi = X\beta^*$, $i \in (S^*)^c \cap S,i' \in S^c \cap S^*, j \in S,$ and  $j' \in S^c \cap (S^*)^c$.  The matrix $\Theta_S(i,i',j,j')$ is defined as
\begin{align}
\Theta_S(i,i',j,j') = \Pi^{\perp}[A_i,i'] - \Pi^{\perp}[A_j,j'] \,. \nonumber 
\end{align}
Using similar methods as in Appendix~\ref{app:EqualCaseTwo}, we can write down an upper bound for $\frac{\Pr({\cal F}_S^c)}{k(p-k)}$ as follows:
\begin{align*}
\mathbbm{1}_{\R^+}\left( \min_{i,i'}\max_{j,j'} \left[ \xi^T \Theta_S(i,i',j,j') \xi + 2 \delta_n^2 \sigma^2 + 2 \sigma ||  \Theta_S(i,i',j,j')^T \xi ||_2 \delta_n \right] \right) 
+ 4e^{-\delta_n^2/2} \,,
\end{align*}
where the $k(p-k)$ arises because ${\cal F}_{S}$ is defined by a term that takes a maximum over $k(p-k)$ possible number of elements.  In Section~\ref{subsec:appprofeq1}, we show that for some $C > 0$ and $\nu_d$ defined in (\ref{eq:nud}),
\begin{align}
\min_{i,i'}\max_{j,j'} &\left[ \xi^T \Theta_S(i,i',j,j') \xi + 2 \delta_n^2 \sigma^2 + 2 \sigma ||  \Theta_S(i,i',j,j')^T \xi ||_2 \delta_n \right] \nonumber \\
& \le C \left[ 
(\nu_d - 1) + \frac{2\sigma\delta_n}{ \sqrt{\frac{dn}{d+1} } \rho_{2k} \beta_{\min}} \left(\sqrt{\nu_{d}} + \sqrt{\frac{2}{\rho_{k-1,0}}}\right) + \frac{2\delta_n^2}{\frac{d}{d+1} n \rho_{2k}^2 \beta_{\min}^2/2}
\right] \label{eq:int1} \,.
\end{align}
Choosing $\delta_n^2 = c_1^2 \frac{d}{d+1} n \rho_{2k}^2 \beta_{\min}^2$, we can upper bound $\Pr({\cal F}_S^c)$ as
\[
k(p-k)\mathbbm{1}_{\R^+} \left( 
(\nu_d - 1) + 2 \sigma c_1 \left( \sqrt{\nu_d} + \sqrt{\frac{2}{\rho_{k-1,0}}} \right)
+ 2(c_1 \sigma)^2 \ge 0 \right) + 4k(p-k) e^{-c_1^2\frac{d}{d+1} n \rho_{2k}^2 \beta_{\min}^2} \,.
\]
If $(\nu_d - 1) + 2 \sigma c_1 \left( \sqrt{\nu_d} + \sqrt{\frac{2}{\rho_{k-1,0}}} \right) + 2(c_1 \sigma)^2< 0$, then $\Pr({\cal F}_S^c) \le 4 k (p-k) e^{-c_1^2\frac{d}{d+1} n \rho_{2k}^2 \beta_{\min}^2/2}$.  Plugging into (\ref{eq:tttt1}), we have
\begin{align}
\Pr(\widehat{S} = S^*) &\ge 1 - \Pr({\cal E}_{S^*}) - 4re^{-c^2 n \beta_{\min}^2 \rho_{2k}^2/2} - 4r k(p-k)  e^{-c^2 \frac{d}{d+1} n \beta_{\min}^2 \rho_{2k}^2/2} \,, \nonumber \\
&\overset{(a)}{\ge} 1 - \Pr({\cal E}_{S^*}) - 5rk(p-k)  e^{-c^2 n \beta_{\min}^2 \rho_{2k}^2/4} \nonumber \,.
\end{align}
In (a), we use the simple bound of $d/(d+1) \ge 1/2$ for $d\ge 1$ and let $c_1 = c$, where $0 < c^2/2 < 1/(18\sigma^2)$.  From Lemma~\ref{lemma:upperR}, $r \le 2^k$, so it is sufficient to choose $n > (2k+\log(k(p-k))/(c^2\beta_{\min}^2 \rho_{2k}^2/4)$ for $\Pr(\widehat{S} = S^*) \rightarrow 1$.

\subsection{Proof of Equation~(\ref{eq:int1})}
\label{subsec:appprofeq1}
Using Lemma~\ref{lemma:blockdecomp}, we have
\begin{align}
\Theta_S(i,i',j,j') &= \underbrace{\Pi^{\perp}[A_{ij}] X_{ij'} \left( X_{ij'}^T \Pi^{\perp}[A_{ij}] X_{ij'}\right)^{-1} X_{ij'}^T \Pi^{\perp}[A_{ij}]}_{P_{ij'}} \nonumber \\
&\qquad - \underbrace{ \Pi^{\perp}[A_{ij}] X_{i'j} \left( X_{i'j}^T \Pi^{\perp}[A_{ij}] X_{i'j}\right)^{-1} X_{i'j}^T \Pi^{\perp}[A_{ij}] }_{P_{i'j}} \,. \nonumber 
\end{align}
Using the notation $Z_{l} = X_l^T \Pi^{\perp}[A_{ij}] X_l$, it is easy to see that the following holds:
\begin{align}
\Lambda_{\min}\left(X_{ij'}^T \Pi^{\perp}[A_{ij}] X_{ij'}\right)
&= (\|Z_{i}\|_2^2 + \|Z_{j'}\|_2^2 - \sqrt{ (\|Z_{i}\|_2^2 - \|Z_{j'}\|_2^2)^2 - 4 (Z_{i}^T Z_{j'})^2 })/2 \nonumber \\
&\ge  \min\{ \|Z_{i'}\|_2^2 , \|Z_{j}\|_2^2 \} \,, \nonumber \\ 
\Lambda_{\max}\left(X_{ij'}^T \Pi^{\perp}[A_{ij}] X_{ij'}\right)
&= (\|Z_{i}\|_2^2 + \|Z_{j'}\|_2^2 + \sqrt{ (\|Z_{i}\|_2^2 - \|Z_{j'}\|_2^2)^2 - 4 (Z_{i}^T Z_{j'})^2 })/2 \nonumber \\
&\le \max \{ \|Z_{i}\|_2^2 , \|Z_{j'}\|_2^2 \} \le n\,. \nonumber 
\end{align}
Next, we have the following upper bound for $\| P_{ij'} \xi \|_2^2$:
\begin{align}
\| P_{ij'} \xi \|_2^2 &\le \left\| \left( X_{ij'}^T \Pi^{\perp}[A_{ij}] X_{ij'}\right)^{-1} \right\|_2 \cdot \| X_{ij'}^T \Pi^{\perp}[A_{ij}] \xi \|_2^2 \,, \nonumber \\
&\le \frac{\| X_{ij'}^T \Pi^{\perp}[A_{ij}] \xi \|_2^2}
{\Lambda_{\min}\left( X_{ij'}^T \Pi^{\perp}[A_{ij}] X_{ij'}\right) }
\le \frac{ (X_{i}^T \Pi^{\perp}[A_{ij}] \xi )^2 + (X_{j'}^T \Pi^{\perp}[A_{ij}] \xi )^2}
{ \min \{ \| \Pi^{\perp}[A_{ij}] X_{i}\|_2^2, \|\Pi^{\perp}[A_{ij}] X_{j'}\|_2^2 \}} \,. \nonumber 
\end{align}
Using similar steps as in Section~\ref{subsec:appprofeq}, we have
\begin{align*}
(X_{i}^T \Pi^{\perp}[A_{ij}] \xi )^2 &= 
\left\| X_{i}^T \Pi^{\perp}[A_{ij}] X_{\bar{S}_j} (X_{\bar{S}_j}^T \Pi^{\perp}[A_{ij}] X_{\bar{S}_j})^{-1} \right\|_1^2 \cdot \| X_{\bar{S}_j}^T \Pi^{\perp}[A_{ij}] \xi \|_{\infty}^2 \,, \\
(X_{j'}^T \Pi^{\perp}[A_{ij}] \xi )^2 &= 
\left\| X_{j'}^T \Pi^{\perp}[A_{ij}] X_{\bar{S}_j} (X_{\bar{S}_j}^T \Pi^{\perp}[A_{ij}] X_{\bar{S}_j})^{-1} \right\|_1^2 \cdot \| X_{\bar{S}_j}^T \Pi^{\perp}[A_{ij}] \xi \|_{\infty}^2 \,,
\end{align*}
where $\bar{S}_j = \{ S^* \backslash S \} \cup \{j\}$.  Defining $\nu_{i,j,j',S}$ as
\[
\nu_{i,j,j',S} = \frac{\left\| X_{i}^T \Pi^{\perp}[A_{ij}] X_{\bar{S}_j} (X_{\bar{S}_j}^T \Pi^{\perp}[A_{ij}] X_{\bar{S}_j})^{-1} \right\|_1^2 + \left\| X_{j'}^T \Pi^{\perp}[A_{ij}] X_{\bar{S}_j} (X_{\bar{S}_j}^T \Pi^{\perp}[A_{ij}] X_{\bar{S}_j})^{-1} \right\|_1^2}{\min \{ \| \Pi^{\perp}[A_{ij}] X_{i}\|_2^2, \|\Pi^{\perp}[A_{ij}] X_{j'}\|_2^2 \}/n} \,,
\]
we have
$
\| P_{ij'} \xi \|_2^2 \le \nu_{i,j,j',S}  \| X_{\bar{S}_j}^T \Pi^{\perp}[A_{ij}] \xi \|_{\infty}^2/n$.   Next, upper and lower bounds for $\|P_{i'j} \xi\|_2^2$ are given as follows:
\begin{align}
\|P_{i'j} \xi\|_2^2 &\le \frac{(X_{i'}^T \Pi^{\perp}[A_{ij}] \xi )^2 + (X_{j}^T \Pi^{\perp}[A_{ij}] \xi )^2}{\min \{ \| \Pi^{\perp}[A_{ij}] X_{i'}\|_2^2, \|\Pi^{\perp}[A_{ij}] X_{j}\|_{\infty}^2 \}} \le \frac{2 \| X_{\bar{S} \cup j} \Pi^{\perp}[A_{ij}] \xi \|_{\infty}^2}{n \rho_{k-1,0}} \,, \nonumber \\
\max_{i' \in \bar{S}} \|P_{i'j} \xi\|_2^2 & \ge \max_{i' \in \bar{S}} \| X_{i'j} \Pi^{\perp}[A_{ij}] \xi \|_{2}^2/n = \| X_{\bar{S}}^T \Pi^{\perp}[A_{ij}] \xi \|_{\infty}^2/n + (X_{j}^T \Pi^{\perp}[A_{ij}] \xi)^2/n 
\nonumber \\
&\ge \|X_{\bar{S}_j} \Pi^{\perp}[A_{ij}] \xi \|_{\infty}^2 / n \,. \nonumber 
\end{align}
We can now upper bound $\xi^T \Theta_S(i,i',j,j') \xi$ and $\| \Theta_S(i,i',j,j') \xi\|$:
\begin{align*}
\min_{i' \in \bar{S}} \xi^T \Theta_S(i,i',j,j') \xi &= \| P_{ij'} \|_2^2 - \max_{i' \in \bar{S}}  \| P_{i'j} \|_2^2 \,, \\
&\le \nu_{i,j,j',S} \| X_{\bar{S}_j}^T \Pi^{\perp}[A_{ij}] \xi\|_{\infty}^2/n - \| X_{\bar{S}_j}^T \Pi^{\perp}[A_{ij}] \xi\|_{\infty}^2/n \,, \\
\| \Theta_S(i,i',j,j') \xi\| &\le \| P_{ij'} \xi \|_2 + \| P_{i'j} \xi \|_2 \,, \\
&\le \sqrt{\nu_{i,j,j',S}}  \| X_{\bar{S}_j}^T \Pi^{\perp}[A_{ij}] \xi\|_{\infty}/\sqrt{n} + \frac{\sqrt{2} \| X_{\bar{S} \cup j} \Pi^{\perp}[A_{ij}] \xi \|_{\infty}}{\sqrt{n \rho_{k-1,0}}} \,.
\end{align*}
Putting everything together, we can upper bound the expression on the left of (\ref{eq:int1}) as follows:
\begin{align}
& \frac{\| X_{\bar{S}_j}^T \Pi^{\perp}[A_{ij}] \xi\|_{\infty}^2}{n} \left[ (\nu_{i,j,j',S} - 1) + \frac{2\sigma \delta_n \sqrt{n}}{ \| X_{\bar{S}_j}^T \Pi^{\perp}[A_{ij}] \xi\|_{\infty}} \left( \sqrt{\nu_{i,j,j',S}} + \frac{\sqrt{2} }{\sqrt{\rho_{k-1,0}}}\right) \right.  \nonumber \\
 &\hspace{8cm} + \left.\frac{2 \delta_n^2 n}{ \| X_{\bar{S}_j}^T \Pi^{\perp}[A_{ij}] \xi\|_{\infty}^2} \right] \nonumber \\
 & \le C \left[
 (\nu_{i,j,j',S} - 1) + \frac{2\sigma\delta_n }{ \sqrt{\frac{dn}{d+1}} \rho_{2k} \beta_{\min}} \left(\sqrt{\nu_{i,j,j',S}} + \frac{\sqrt{2} }{ \sqrt{\rho_{k-1,0}}}\right) + \frac{2\delta_n^2}{\frac{d}{d+1} n \rho_{2k}^2 \beta_{\min}^2}
 \right] \,, \nonumber 
\end{align}
where $C > 0$ and we use the inequality
$ \| X_{\bar{S}_j}^T \Pi^{\perp}[A_{ij}] \xi \|_{\infty}^2 \ge \frac{d}{d+1} n^2 \rho_{2k}^2 \beta_{\min}^2$.  Taking the minimum over $i \in (S^*)^c$ and maximum over $j \in S$ and $j' \in S^c \cap (S^*)^c$, we obtain the desired upper bound of
\[
C \left[ 
(\nu_d - 1) + \frac{2\sigma\delta_n}{ \sqrt{\frac{dn}{d+1}} \rho_{2k} \beta_{\min}} \left(\sqrt{\nu_{d}} + \frac{\sqrt{2} }{\sqrt{\rho_{k-1,0}}}\right) + \frac{2 \delta_n^2}{{\frac{dn}{d+1}\rho_{2k}^2 \beta_{\min}}}
\right] \,,
\] 
where $\nu_d$ is defined in (\ref{eq:nud}).

\section{Assorted Useful Lemmata}

In this Section, we collect some useful lemmata we used in the proofs above.  For the first three lemmas, let $w$ be a random vector with parameter $\sigma$ so that the entries of $w$ are i.i.d. zero mean sub-Gaussian random variables with parameter $\sigma$. 

\begin{lemma}[\citet{Hsu2012Prob}]
\label{lemma:firstSG}
If $A$ is a rank $\ell$ projection matrix, then $\Pr(\|Aw\|_2^2/\sigma^2 - \ell > 2 \sqrt{\ell t} + 2t) \le e^{-t}$ for all $t \ge 0$.
\end{lemma}
\begin{lemma}
\label{lemma:secondSG}
If $A_1$ and $A_2$ are rank $\ell$ projection matrices, then for $t \ge 1$,
\[
\Pr\left( \left| \|A_1 w\|_2^2 - \|A_2 w\|_2^2\right| \ge 8 \sigma^2 \ell t  \right) \le 2e^{-\ell t} \,.
\]
\end{lemma}
\begin{proof}
Using the triangle inequality and the union bound, we have
\begin{align*}
\Pr\left( \left| \|A_1 w\|_2^2 - \|A_2 w\|_2^2\right|/\sigma^2  > 2x \right) &\le 
\Pr\left( \left| \|A_1 w\|_2^2 /\sigma^2 - \ell \right|  > x  \right)
+ \Pr\left( \left| \|A_2 w\|_2^2 /\sigma^2- \ell \right|  > x  \right) \\
& \le 2 \P\left( \left| \|A_1 w\|_2^2 /\sigma^2 - \ell \right|  > x  \right) \,.
\end{align*}
Analyzing $\Pr\left( \left| \|A_1 w\|_2^2/ \sigma^2 - \ell \right|   > x  \right)$, we have
\begin{align*}
\Pr\left( \left| \|A_1 w\|_2^2/\sigma^2 - \ell \right| > x  \right) &\le 
\Pr\left( \|A_1 w\|_2^2 /\sigma^2 - \ell   > x  \right) + \Pr\left(  \|A_1 w\|_2^2/\sigma^2 - \ell  < -x  \right) \\
&\le \Pr\left( \|A_1 w\|_2^2 /\sigma^2 - \ell   > x  \right), \text{ when $x > \ell$} \,.
\end{align*}
Substituting, we have that for $x > \ell$,
\begin{align*}
\Pr\left( \left| \|A_1 w\|_2^2 - \|A_2 w\|_2^2\right|/\sigma^2 > 2x \right) \le 
2  \Pr\left( \|A_1 w\|_2^2 /\sigma^2 - \ell   > x  \right) \,.
\end{align*}
Let $x = 2\sqrt{\ell \cdot \ell t} + 2 \ell t $.  Since $4t \ge 2 \sqrt{t} + 2t$, we obtain the desired result using Lemma~\ref{lemma:firstSG}.
\end{proof}
\begin{lemma}
\label{lemma:thirdSG}
For any vector $a \in \R^n$, $\Pr(|a^T w | > t) \le 2e^{-t^2/(2\|a\|_2^2 \sigma^2)}$.
\end{lemma}
\begin{lemma}
\label{lemma:blockdecomp}
$\Pi[A_1,A_2] = \Pi[A_1] + \Pi^{\perp}[A_1] X_{A_2} ( X_{A_2}^T \Pi^{\perp}[A_1] X_{A_2})^{-1} X_{A_2}^T \Pi^{\perp}[A_1]
$
\end{lemma}
\begin{proof}
Follows from the block-inversion formula.
\end{proof}
In the next lemma, let $\Lambda_{\min}(W)$ and $\Lambda_{\max}(W)$ denote the minimum and maximum eigenvalues of a matrix $W$.
\begin{lemma}
\label{lemma:eigenvalues}
Let $A_1,A_2 \subset [p]$ that are disjoint.  We have the following results:
\begin{align*}
\Lambda_{\min}( (X_{A_1\cup A_2}^T X_{A_1 \cup A_2})^{-1}) &\le \Lambda_{\min}( ( X_{A_1}^T \Pi^{\perp}[A_2] X_{A_1})^{-1}) \\
\Lambda_{\max}( (X_{A_1}^T \Pi^{\perp}[A_2] X_{A_1})^{-1}) &\le \Lambda_{\max}(( X_{A_1\cup A_2}^T X_{A_1 \cup A_2})^{-1})
\end{align*}
\end{lemma}
\begin{proof}
Follows from the block-inversion formula and extensions of the Cauchy's interlacing theorem.
\end{proof}

\end{document}